\documentclass[11pt]{article}
\usepackage{amscd}
\usepackage{amsmath}

\usepackage{amsfonts}
\usepackage{setspace}
\usepackage{url}
\usepackage{amssymb}
\usepackage{stmaryrd}
\usepackage{mathrsfs}
\usepackage{latexsym,bm}
\usepackage{yfonts}
\usepackage{amsmath,amssymb}
\usepackage[all,poly,knot]{xy}
\usepackage{mathtools}

\usepackage[hang,small,bf]{caption}

\usepackage{amsthm}

\usepackage{CJK}     
\usepackage{geometry} 
\usepackage{fancyhdr} 
\usepackage{fancyvrb} 
\usepackage{fancybox} 
\usepackage{amsmath,amsfonts,amssymb,graphicx}  
\usepackage{subfigure}  
\usepackage{indentfirst} 
\usepackage{bm}     
\usepackage{multicol}  
\usepackage{picins}   
\usepackage{abstract}  
\usepackage{anysize} 
\usepackage{listings}
\usepackage{color} 
\usepackage{units} 
\usepackage{tabularx} 
\usepackage{mathrsfs} 
\usepackage{bm}
\usepackage{geometry}
\newtheorem{Theorem}{Theorem}[section]

\newtheorem{Proposition}[Theorem]{Proposition}
\newtheorem{Corollary}[Theorem]{Corollary}

\newtheorem{Example}[Theorem]{Example}

\catcode`@=11 \@addtoreset{equation}{section} \catcode`@=12
\title{ Orthogonal inner product graphs of odd characteristic and their automorphisms\thanks{
Supported by the National Natural Science Foundation of China (No.\:11371343), Guangxi science and technology development foundation (No.\:1599005-2-13), the Guangxi Science
Foundation (No.\:0832107, 0640070, 0991102) and the Scientific Research Foundation of Guangxi Educational Committee.}}
\author{Shouxiang ZHAO$^{\rm a,b}$, Hengbin ZHANG$^{\rm c}${\thanks{Corresponding author.
			E-mail: shouxiangzhao@163.com}, Jizhu NAN$^{\rm a}$,
Gaohua TANG$^{\rm d}$ }}
\date{$^{\rm a}${\small School of Mathematical Sciences, Dalian University of
Technology, Dalian 116024, P.R. China}\\
$^{\rm b}${\small Department of Mathematics and Computer
Science, Guilin Normal College, Guilin 541001, P.R. China}\\
$^{\rm c}${\small College of Mathematical Science, Yangzhou University, Yangzhou 225002, P.R. China}\\
$^{\rm d}${\small School of Mathematical Sciences, Guangxi Teachers Education
University, Nanning, 530023, P.R. China}}

\small
\begin{document}
\maketitle{}{ \textbf{Abstract:}} Let $\mathbb{F}_q$ be a finite field of odd characteristic and $2\nu+\delta\geq2$ an integer number with $\delta=0,1$ or $2$. The orthogonal inner product graph $Oi\big(2\nu+\delta,q\big)$ over $\mathbb{F}_q$ is defined and the automorphism groups of $Oi\big(2\nu+\delta,q\big)$ are determined. We show that $Oi\big(2\nu+\delta,q\big)$ is a disconnected graph if $2\nu+\delta=2$; otherwise it is not. Moreover, we have two necessary and sufficient conditions for two vertices of $Oi\big(2\nu+\delta,q\big)$ and two edges of $Oi\big(2\nu+\delta,q\big)$ respectively are in the same orbit under the action of the automorphism group of $Oi\big(2\nu+\delta,q\big).$



\vskip3mm

\section{Introduction}
\setlength{\parindent}{2em}

Let $\mathbb{F}_q$ be a finite field of odd characteristic and $n\geq 1$ an integer number. Let $\mathbb{F}_q^{(n)}=\{(a_1, \ldots,a_n) : a_i\in \mathbb{F}_q\ {\rm for}\ i=1,\ldots,n\}$ be the $n$-dimensional vector space over $\mathbb{F}_q$. For the $m$-dimensional vector subspace $P$ of $\mathbb{F}_q^{(n)}$, any $m\times n$ matrix whose rows, $\alpha_1,\ldots, \alpha_m$, form a basis of the space $P$, is called a \emph{matrix representation} of the space $P$, which is also denoted by $P$. To simplify notation, sometimes we write $[\alpha_1,\ldots,\alpha_m]$ instead of the matrix $P$. Let $e_1=(1,0,\ldots,0), e_2=(0,1,\ldots,0), \ldots, e_n=(0,0,\ldots,1)$. Then $e_i,i=1,2,\ldots,n$, form a basis of $\mathbb{F}_q^{(n)}$. The set of all non-zero elements of $\mathbb{F}_q$, denoted by $\mathbb{F}_q^{*}$, form a cyclic group under field multiplication, called the \emph{multiplicative group} of $\mathbb{F}_q$. The set of all square elements of $\mathbb{F}_q^*$, denoted by $\mathbb{F}_q^{*2}$, is  a subgroup of $\mathbb{F}_q^*$. Let $z$ be a fixed non-square element in $\mathbb{F}_q^*$ and
\begin{displaymath}
S_{2\nu+\delta,\Delta}=
\left( \begin{array}{ccc}
0 & I^{(\nu)} & \\
I^{(\nu)}& 0 & \\
 & & \Delta
\end{array} \right),
\end{displaymath}
be four $(2\nu+\delta)\times (2\nu+\delta)$ nonsingular symmetric matrices, where~
\begin{displaymath}\Delta= \left\{ \begin{array}{ll}
\emptyset,   & \textrm{if $\delta=0,$}\\
(1)\ {\rm or}\ (z),   & \textrm{if $\delta=1,$}\\
{\rm diag}(1,-z), & \textrm{if $\delta=2.$}
\end{array} \right.
\end{displaymath}
To simplify notations, we omit $\emptyset$ when $\delta=0$ and write $S$ for $S_{2\nu+\delta,\Delta}$. A $(2\nu+\delta)\times (2\nu+\delta)$ nonsingular matrix $T$ over $\mathbb{F}_q$ is called \emph{orthogonal} with respect to $S$ if $TS\, ^{t}\!T=S$. The set of all $(2\nu+\delta)\times (2\nu+\delta)$ orthogonal matrices form a group under matrix multiplication, called the \emph{orthogonal group} of degree $2\nu+\delta$ with respect to $S$ over $\mathbb{F}_q$ and denoted by $O_{2\nu+\delta}\big(\mathbb{F}_q\big)$. The factor group $O_{2\nu+\delta}\big(\mathbb{F}_q\big) \big/ \big\{I, -I\big\}$ is called the \emph{projective orthogonal group} of degree $2\nu+\delta$ over $\mathbb{F}_q$ and denoted by $PO_{2\nu+\delta}\big(\mathbb{F}_q\big)$.
There is an action of $O_{2\nu+\delta}(\mathbb{F}_q)$ on $\mathbb{F}_q^{(2\nu+\delta)}$ defined as follows:
\begin{eqnarray*}
\mathbb{F}_q^{(2\nu+\delta)}\times O_{2\nu+\delta}(\mathbb{F}_q) &\longrightarrow& \mathbb{F}_q^{(2\nu+\delta)}\\
((x_1,\ldots,x_{2\nu+\delta}),T) &\longrightarrow& (x_1,\ldots,x_{2\nu+\delta})T.
\end{eqnarray*}
The vector space $\mathbb{F}_q^{(2\nu+\delta)}$ together with this group action is said to be the $(2\nu+\delta)$-dimensional \emph{orthogonal space} over $\mathbb{F}_q.$ The map: $\mathbb{F}_q^{(2\nu+\delta)}\times \mathbb{F}_q^{(2\nu+\delta)}\to \mathbb{F}_q$ given by $(\alpha,\beta)\mapsto\alpha S\,{^t\!\beta}$ is said to be an \emph{orthogonal inner product} of $\mathbb{F}_q^{(2\nu+\delta)}$ with respect to $S$. For any subspace $P$ of $\mathbb{F}_q^{(2\nu+\delta)}$, the set $\{\alpha \in \mathbb{F}_q^{(2\nu+\delta)}:\alpha S\, ^{t}\!\beta = 0$ for all $\beta \in P\}$, denoted by $P^\perp$, is said to be the \emph{dual subspace} of $P$ with respect to $S$.

\medskip
In general, determining all the automorphisms of a graph is an important but difficult problem both in graph theory and in algebraic theory. In 2006, Tang and Wan \cite{ztzw} introduced the concept of the symplectic graph over a finite field. They proved that the symplectic graph is strongly regular and determined its automorphism group. In \cite{zgzw,zwkz,zwkz2}, the authors introduced the concepts of the orthogonal graph and the unitary graph over a finite field and determined their automorphism groups, respectively. In 2014, Wong et al.\:\cite{dwxmjz} introduced the concept of the zero-divisor graph based on rank one upper triangular matrices and also determined its automorphism group. Motivated by previous studies, we introduce a new graph on non-trivial orthogonal spaces over a finite field for studying the interplay between properties of orthogonal subspaces and the structure of graphs.
\medskip

The \emph{orthogonal inner product graph} with relative to $S$ over $\mathbb{F}_q$, denoted by $Oi\big(2\nu+\delta,q\big)$, is the graph defined on all non-trivial orthogonal subspaces of $\mathbb{F}_q^{(2\nu+\delta)}$ with an edge between vertices $A$ and $B,$ denoted by $A$---$B,$ if and only if $AS\,^{t}\!B=0$. The set of all vertices and all edges of $Oi\big(2\nu+\delta,q\big)$ are denoted by $V\big(Oi\big(2\nu+\delta,q\big)\big)$ and $E\big(Oi\big(2\nu+\delta,q\big)\big),$ respectively$.$ Of course, if $A$---$B$ is in $E\big(Oi\big(2\nu+\delta,q\big)\big)$, then $B$ is a subspace of $A^\perp$ and $A$ is a subspace of $B^\perp$.
\medskip

In section 2, we show that $Oi\big(2\nu+\delta,q\big)$ is connected with diameter $4$ when $2\nu+\delta\geq3$ and also show that $Oi\big(2\nu_1+\delta_1,q_1\big)\cong Oi\big(2\nu_2+\delta_2,q_2\big)$ if and only if $\nu_1=\nu_2,$ $\delta_1=\delta_2$ and $q_1=q_2$. In section 3, we obtain two necessary and sufficient conditions: (1) two vertices of $V\big(Oi\big(2\nu+\delta,q\big)\big)$ are in the same orbit under the action of ${\rm Aut}\big(Oi\big(2\nu+\delta,q\big)\big)$ if and only if they are the same type subspace; (2) two edges $g$ and $f$ of $E\big(Oi\big(2\nu+\delta,q\big)\big)$ are in the same orbit under the action of ${\rm Aut}\big(Oi\big(2\nu+\delta,q\big)\big)$ if and only if the following conditions hold: (\romannumeral1) one vertex of $g$ and one vertex of $f$ are the same type subspace, (\romannumeral2) the other vertex of $g$ and the other vertex of $f$ are the same type subspace, and (\romannumeral3) the sum of the two vertices of $g$ and the sum of the two vertices of $f$ are also the same type subspace. Furthermore, we show that ${\rm Aut}\big(Oi\big(2\nu+\delta,q\big)\big)= PO_{2\nu+\delta}\big(\mathbb{F}_q\big)\cdot E_{2\nu+\delta}$, where $E_{2\nu+\delta}$ is a subgroup of ${\rm Aut}\big(Oi\big(2\nu+\delta,q\big)\big)$.

\section{\boldmath Some properties of orthogonal inner graphs}

In \cite{jabusrm}, a graph $G$ is \emph{connected} if there is a path between any two vertices of $G$; otherwise $G$ is disconnected. The \emph{distance} between vertices $a$ and $b$ in $G$ is the number of edges in a shortest path between $a$ and $b$ and denoted by $d(a,b)$. The \emph{diameter} of G is the greatest distance between any two vertices of $G$ and denoted by ${\rm diam}(G)$. The \emph{degree} of $x\in G$ is defined to be the number of edges of the form $x$---$y$ in $G$ and denoted by $\deg(x)$.


\begin{Theorem} Let $\mathbb{F}_q$ be a finite field. Then $Oi\big(2\nu+\delta,q\big)$ is a connected graph if and only if $2\nu+\delta\geq3$. Moreover, if $Oi\big(2\nu+\delta,q\big)$ is a connected graph, then diam$\big(Oi\big(2\nu+\delta,q\big)\big)=4$.
\end{Theorem}

\begin{proof}  In order to prove the theorem, we need only consider three cases: $\delta=0,1$ and $2$. Without loss of generality, we need only consider the case of $\delta=0$. Other cases are similar to prove. Suppose that $\delta=0$.

Let $\delta=0$ and $\nu=1$. Clearly, orthogonal subspaces $[e_1]$ and $[e_2]$ of $\mathbb{F}_q^{(2)}$ are different. Since $[e_1]=[e_1]^{\bot}$, $[e_2]=[e_2]^{\bot}$ and $[e_1,e_2]^{\bot}=0$, we have d$([e_1],[e_2])=\infty$ and so $Oi\big(2,q\big)$ is not connected.

Let  $\delta=0$, $\nu\geq2$ and $A,B\in V\big(Oi\big(2\nu,q\big)\big)$. If $A$ and $B$ are adjacent, then $A$---$B$ is a path of length 1. So in the remainder of the proof we assume that $A$ and $B$ are not adjacent. Then there exist $[\alpha],[\beta]\in V\big(Oi\big(2\nu,q\big)\big)$ such that $[\alpha]\subseteq A^\perp$ and $[\beta]\subseteq B^\perp$. If $[\alpha]$ and $[\beta]$ are adjacent, then $A$--- $[\alpha]$---$[\beta]$---$B$ is a path of length 3. If $[\alpha]$ and $[\beta]$ are not adjacent, then $\dim([\alpha,\beta]^{\perp})\geq2\nu-2>2$ and there exists $[\gamma]\subseteq [\alpha,\beta]^{\perp}$ such that $A$---$[\alpha]$---$[\gamma]$---$[\beta]$---$B$ is a path of length 4. Therefore $Oi\big(2\nu,q\big)$ is connected and diam$\big(Oi\big(2\nu,q\big)\big)\leq4$. Let $A=[e_1,\ldots,e_{2\nu-1}]$ and $B=[e_1,\ldots,e_{\nu-1},e_{\nu+1},\ldots,e_{2\nu}]$ with $A^{\bot}=[e_\nu]$ and $B^{\bot}=[e_{2\nu}]$. Then $A$---$[e_\nu]$---$[e_1]$---$[e_{2\nu}]$---$B$ is a shortest path of length 4 between $A$ and $B$ and so diam$\big(Oi\big(2\nu,q\big)\big)=4$.
\end{proof}

The set of all vertices of dimension 1 in $Oi\big(2\nu+\delta,q\big)$ can induce a subgraph of $Oi\big(2\nu+\delta,q\big)$, denoted by $Oi\big(2\nu+\delta,1,q\big)$, with an edge between vertices $A$ and $B$ if and only if $AS\,^{t}\!B=0.$ Of course, if $X\in V\big(Oi\big(2\nu+\delta,1,q\big)\big)$, then $X$---$X^{\perp}$ is in $E\big(Oi\big(2\nu+\delta,q\big)\big)$ with $\deg(X^{\perp})=1$. Next, we introduce some notation and terminology  from \cite{zwG} that will be used in this paper$.$ We will denote by $\mathcal{M}(m,2s+\gamma,s,\Gamma;2\nu+\delta,\Delta)$ the set of subspaces of type $(m,2s+\gamma,s,\Gamma)$ of $\mathbb{F}_q^{(2\nu+\delta)}$ with respect to $S$ and write $N(m,2s+\gamma,s,\Gamma;2\nu+\delta,\Delta)= |\mathcal{M}(m,2s+\gamma,s,\Gamma;2\nu+\delta,\Delta)|.$ Let $P$ be a fixed subspace of type $(m,2s+\gamma,s,\Gamma)$ in $\mathbb{F}_q^{(2\nu+\delta)}.$ We will denote by $\mathcal{M}_P(m_1,2s_1+\gamma_1,s_1,\Gamma_1;m,2s+\gamma,s,\Gamma;2\nu+\delta,\Delta)$ the set of subspaces of type $(m_1,2s_1+\gamma_1,s_1,\Gamma_1)$ contained in $P$ and write $N(m_1,2s_1+\gamma_1,s_1,\Gamma_1;m,2s+\gamma,s,\Gamma;2\nu+\delta,\Delta)= |\mathcal{M}_P(m_1,2s_1+\gamma_1,s_1,\Gamma_1;m,2s+\gamma,s,\Gamma;2\nu+\delta,\Delta)|.$ Then
$$V\big(Oi\big(2\nu+\delta,1,q\big)\big)= \mathcal{M}(1,0,0;2\nu+\delta,\Delta)\bigcup\mathcal{M}(1,1,0,1;2\nu+\delta,\Delta) \bigcup\mathcal{M}(1,1,0,z;2\nu+\delta,\Delta)$$
and $|V\big(Oi\big(2\nu+\delta,1,q\big)\big)|=\frac{q^{2\nu+\delta}-1}{q-1}$ by \cite[Theorem 6.26]{zwG}.

\vskip 0.3cm

\begin{Theorem} Let $\mathbb{F}_{q_1}$ and $\mathbb{F}_{q_2}$ be finite fields. Then $Oi\big(2\nu_1+\delta_1,q_1\big)\cong Oi\big(2\nu_2+\delta_2,q_{2}\big)$ if and only if $\nu_1=\nu_2$, $\delta_1=\delta_2$ and $q_1=q_2$.
\end{Theorem}
\begin{proof} First suppose that $\nu_1=\nu_2$, $\delta_1=\delta_2$ and $q_1=q_2$. Then $Oi\big(2\nu_1+\delta_1,q_1\big)\cong Oi\big(2\nu_2+\delta_2,q_{2}\big)$, completing the proof in one direction.

For the other direction, suppose that $Oi\big(2\nu_1+\delta_1,q_1\big)\cong Oi\big(2\nu_2+ \delta_2,q_{2}\big)$. We need only consider $Oi\big(2\nu_i+\delta_i,q_i\big)$ for $i=1,2$. Then $(q_1^{2\nu_1+ \delta_1}-1)/(q_1-1)= (q_2^{2\nu_2+\delta_2}-1)/(q_2-1)$. Take $M_i$ being a maximum complete subgraph (without considering circles) of $Oi\big(2\nu_i+\delta_i,1,q_i\big)$ such that $XS_{2\nu_i+\delta_i} \, ^{t}\!Y=0$ for all different $X,Y\in M_i$. Let $N_i=\{ Z\in M_i:ZS_{2\nu_i+\delta_i} \, ^{t}\!Z\neq0\}$ for $i=1,2$. Then $\nu_1+\delta_1=|M_1|=|M_2|= \nu_2+\delta_2\ {\rm and}\ \delta_1=|N_1|=|N_2|=\delta_2.$ Thus $\nu_1=\nu_2$, $\delta_1=\delta_2$ and $q_1=q_2$.
\end{proof}

\begin{Proposition}  \label{2.3} Let $Oi\big(2\nu+\delta,q\big)$ be the orthogonal inner graph over $\mathbb{F}_{q}$ and $\sigma\in {\rm Aut}\big(Oi\big(2\nu+\delta,q\big)\big)$. Then $\sigma\big(Oi\big(2\nu+\delta,1,q\big)\big)=Oi\big(2\nu+\delta,1,q\big)$.
\end{Proposition}

\begin{proof} Let $X\in V\big(Oi\big(2\nu+\delta,1,q\big)\big)$ and $\sigma\in {\rm Aut}\big(Oi\big(2\nu+\delta,q\big)\big)$. Then there exists $A\in V\big(Oi\big(2\nu+\delta,q\big)\big)$ with $ XS_{2\nu+\delta}\, ^{t}\!A=0$ and deg$(A)=2\nu+\delta-1$. Since $\sigma\in {\rm Aut}\big(Oi\big(2\nu+\delta,q\big)\big)$, it follows that $\sigma(X)S_{2\nu+\delta}\, ^{t}\!(\sigma(A))=0$ and deg$(\sigma(A))=2\nu+\delta-1$. So $\sigma(X)\in V\big(Oi\big(2\nu+\delta,1,q\big)\big)$. Since $\sigma$ is injective and $V\big(Oi\big(2\nu+\delta,1,q\big)\big)$ is finite, we have $\sigma\big(Oi\big(2\nu+\delta,1,q\big)\big)=Oi\big(2\nu+\delta,1,q\big)$.
\end{proof}

It follows that the complement graph of the orthogonal graph in \cite{zgzw} is a proper subgraph of $Oi\big(2\nu+\delta,1,q\big)$. In fact, the vertices set of the orthogonal graph in \cite{zgzw} contains all $1$-dimensional totally isotropic orthogonal subspaces of $\mathbb{F}_q^{(2\nu+\delta)}$. But $Oi\big(2\nu+\delta,1,q\big)$ contains all $1$-dimensional orthogonal subspaces of $\mathbb{F}_q^{(2\nu+\delta)}$.

\begin{Proposition}  \label{2.4} Let $T\in O_{2\nu+\delta}\big(\mathbb{F}_q\big)$ and
$$\sigma_T:V\big(Oi\big(2\nu+\delta,q\big)\big)\longrightarrow V\big(Oi\big(2\nu+\delta,q\big)\big),A\longmapsto AT,$$
for $\delta=0,1$ or $2.$ Then

\noindent${\rm(1)}$ $\sigma_T\in {\rm Aut}\big(Oi\big(2\nu+\delta,q\big)\big),$

\noindent ${\rm(2)}$ For any $T_1,T_2\in O_{2\nu+\delta}\big(\mathbb{F}_q\big)$, $\sigma_{T_1}= \sigma_{T_2}$ if and only if $T_1=\pm T_2$.
\end{Proposition}

\begin{proof} (1) Let $T\in O_{2\nu+\delta}\big(\mathbb{F}_q\big)$. Then $T$ is nonsingular and $\sigma_T$ is a bijection. For any $A,B\in V\big(Oi\big(2\nu+\delta, q\big)\big)$, we know that $AS\, ^{t}\!B=ATS\, ^{t}\!(AT)$, then $A\text{---}B$ if and only if $\sigma_T(A)\text{---}\sigma_T(B)$. Thus we conclude that $\sigma_T\in {\rm Aut}\big(Oi\big(2\nu+\delta,q\big)\big)$.

(2) First of all, if $T_1=\pm T_2\in O_{2\nu+\delta}\big(\mathbb{F}_q\big)$, then it follows that $\sigma_{T_1}= \sigma_{T_2}$. Conversely, suppose that $\sigma_{T_1}=\sigma_{T_2}$. Then we need only consider them acting on $Oi\big(2\nu+\delta,q\big)$. According to Proposition \ref{2.3}, we deduce that $\sigma_{T_i}\big(Oi\big(2\nu+\delta,1,q\big)\big)=Oi\big(2\nu+\delta,1,q\big)$ for $i=1,2$. The following proof is similar to that of \cite[Proposition 3.1]{zgzw}.
\end{proof}

Note that every orthogonal matrix of $O_{2\nu+\delta}(\mathbb{F}_q)$ can induce an automorphism of $Oi\big(2\nu+\delta,q\big)$ and two different orthogonal matrices $T_1$ and $T_2$ induce the same automorphism of $Oi\big(2\nu+\delta,q\big)$ if and only if $T_1=\pm T_2$. Thus $PO_{2\nu+\delta}(\mathbb{F}_q)$ can be regarded as a subgroup of ${\rm Aut}\big(Oi\big(2\nu+\delta,q\big)\big)$.

\begin{Proposition}  \label{2.5} Let $Oi\big(2\nu+\delta,1,q\big)$ be an induced subgraph of $Oi\big(2\nu+\delta,q\big)$ and $\sigma\in {\rm Aut}\big(Oi\big(2\nu+\delta,q\big)\big)$. Then $X$ and $\sigma(X)$ are the same type orthogonal subspace for $X\in V\big(Oi\big(2\nu+\delta,1,q\big)\big)$.
\end{Proposition}
\begin{proof} In order to prove the theorem, we need only consider three cases: $\delta=0,1$ and $2$. Without loss of generality, we need only consider the case of $\delta=0$. Other cases are similar to prove. Let $\delta=0$.

Let $X\in V\big(Oi\big(2\nu+\delta,1,q\big)\big)$. Then the type of $X$ is one of $(1,0,0)$, $(1,1,0,1)$ and $(1,1,0,z)$. If the type of $X$ is $(1,0,0)$, then it is easy to check that the type of $\sigma(X)$ is also $(1,0,0)$. By \cite[Theorem 6.4]{zgzw} and Proposition \ref{2.4}, we know that the sets $\mathcal{M}(1,0,0;2\nu+\delta,\Delta)$, $\mathcal{M}(1,1,0,1;2\nu+\delta,\Delta)$ and $\mathcal{M}(1,1,0,z;2\nu+\delta,\Delta)$ are three orbits of $Oi\big(2\nu+\delta,q\big)$ under the action of $PO_{2\nu+\delta,\Delta}(\mathbb{F}_q)$.

If the type of $X$ is $(1,1,0,1)$, then by \cite[Corollary 6.6]{zgzw} we know that the type of $X^{\perp}$ is
\begin{displaymath}
\left\{ \begin{array}{ll}
(2\nu-1,2\nu-1,\nu-1,1), & \text{if}~-1\in\mathbb{F}_q^{*2},\\
(2\nu-1,2\nu-1,\nu-1,z), & \text{if}~-1\notin\mathbb{F}_q^{*2}.
\end{array} \right.
\end{displaymath}
If the type of $X$ is $(1,1,0,z)$, then by \cite[Corollary 6.6]{zgzw} we know that the type of $X^{\perp}$ is
\begin{displaymath}
\left\{ \begin{array}{ll}
(2\nu-1,2\nu-1,\nu-1,z), & \text{if}~-1\in\mathbb{F}_q^{*2},\\
(2\nu-1,2\nu-1,\nu-1,1), & \text{if}~-1\notin\mathbb{F}_q^{*2}.
\end{array} \right.
\end{displaymath}

Let $$\mathcal{X}=\{Y\in V\big(Oi\big(2\nu,q\big)\big):YS_{2\nu}{^t\!X}=0 {\rm\ and\ the\ type\ of\ Y\ is\ (1,1,0,1)}\}.$$

If the type of $X$ is $(1,1,0,1)$, then by \cite[Theorem 6.33]{zgzw},  we have
\begin{displaymath}
|\mathcal{X}|=\left\{ \begin{array}{ll}
N(1,1,0,1;2\nu-1,2\nu-1,\nu-1,1;2\nu), & \text{if}~-1\in\mathbb{F}_q^{*2},\\
N(1,1,0,1;2\nu-1,2\nu-1,\nu-1,z;2\nu), & \text{if}~-1\notin\mathbb{F}_q^{*2},
\end{array} \right.
\end{displaymath}
\begin{displaymath}
=\left\{ \begin{array}{ll}
N(1,1,0,1;2\nu-1,1), & \text{if}~-1\in\mathbb{F}_q^{*2},\\
N(1,1,0,1;2\nu-1,z), & \text{if}~-1\notin\mathbb{F}_q^{*2}.
\end{array} \right.
\end{displaymath}

If the type of $X$ is $(1,1,0,z)$, then by \cite[Theorem 6.33]{zgzw},  we have
\begin{displaymath}
|\mathcal{X}|=\left\{ \begin{array}{ll}
N(1,1,0,1;2\nu-1,2\nu-1,\nu-1,z;2\nu), & \text{if}~-1\in\mathbb{F}_q^{*2},\\
N(1,1,0,1;2\nu-1,2\nu-1,\nu-1,1;2\nu), & \text{if}~-1\notin\mathbb{F}_q^{*2},
\end{array} \right.
\end{displaymath}
\begin{displaymath}
=\left\{ \begin{array}{ll}
N(1,1,0,1;2\nu-1,z), & \text{if}~-1\in\mathbb{F}_q^{*2},\\
N(1,1,0,1;2\nu-1,1), & \text{if}~-1\notin\mathbb{F}_q^{*2}.
\end{array} \right.
\end{displaymath}

If the type of $X_1$ is $(1,1,0,1)$ and the type of $X_2$ is $(1,1,0,z)$, then by \cite[Theorem 6.26]{zgzw}, we have $|\mathcal{X}_1|\neq|\mathcal{X}_2|$. Thus by Proposition \ref{2.3}, we know that $X$ and $\sigma(X)$ are the same type orthogonal subspace for $X\in V\big(Oi\big(2\nu+\delta,1,q\big)\big)$.
\end{proof}

\section{\boldmath Automorphism groups of orthogonal inner graphs}

In order to determine automorphism groups ${\rm Aut}\big(Oi\big(2\nu+\delta,q\big)\big)$ of orthogonal inner graphs $Oi(2\nu+\delta, \mathbb{F}_q)$, we need only consider three cases of $\delta=0,1$ and $2$.

Firstly, we will discuss the case of $\delta=0$. Let $S=S_{2\nu}$. In what follows, we write $f_i$ for $e_{\nu+i},$ $1\leq i\leq \nu$. 
Then
$$e_i{S_{2\nu}}{^t\!f_i}=1,e_i{S_{2\nu}}{^t\!e_j}=f_i{S_{2\nu}}{^t\!f_j}=0\ {\rm for}\ 1\leq i,j\leq\nu$$
and
$$e_i{S_{2\nu}}{^t\!f_j}=0\ {\rm for}\ i\neq j,\ 1\leq i,j\leq\nu.$$

Let $\varphi_{2\nu}$ be the natural action of ${\rm Aut}(\mathbb{F}_q)$ on the group $\mathbb{F}_q^{*2}\times\mathbb{F}_q^*\times \cdots\times\mathbb{F}_q^*$ defined by
\begin{eqnarray*}
\varphi_{2\nu}:\big(\mathbb{F}_q^{*2}\times\mathbb{F}_q^*\times\cdots\times \mathbb{F}_q^*\big)\times{\rm Aut}(\mathbb{F}_q) &\longrightarrow& \mathbb{F}_q^{*2}\times\mathbb{F}_q^*\times \cdots\times\mathbb{F}_q^*\\
((k_1,k_2,\ldots,k_{\nu}),\pi) &\longrightarrow& (\pi(k_1),\pi(k_2),\ldots,\pi(k_{\nu}))
\end{eqnarray*}
for $\pi\in{\rm Aut}(\mathbb{F}_q)$, $k_1\in\mathbb{F}_q^{*2}$ and $k_2,\ldots,k_{\nu}\in\mathbb{F}_q^*$.
The \emph{semidirect product} of $\mathbb{F}_q^{*2}\times\mathbb{F}_q^*\times \cdots\times\mathbb{F}_q^*$ by ${\rm Aut}(\mathbb{F}_q)$ with respect to $\varphi_{2\nu}$, denoted by $(\mathbb{F}_q^{*2}\times\mathbb{F}_q^*\times \cdots\times\mathbb{F}_q^*)\rtimes_{\varphi_{2\nu}} {\rm Aut}(\mathbb{F}_q)$, is the group consisting of all elements of the form $(k_1,k_2\ldots, k_{\nu},\pi)$, with the multiplication defined by
$$(k_1,k_2,\ldots,k_{\nu},\pi)(k_1^{'},k_2^{'},\ldots,k_{\nu}^{'},\pi^{'})=(k_1\pi(k_1^{'}), k_2\pi(k_2^{'}),\ldots, k_{\nu}\pi(k_{\nu}^{'}),\pi\pi^{'}).$$
Define maps:
\begin{eqnarray*}
\sigma_{\pi}:V\big(Oi\big(2\nu+\delta,q\big)\big) & \longrightarrow & V\big(Oi\big(2\nu+\delta,q\big)\big)\\
\big(a_{ij}\big)_{m\times(2\nu+\delta)}&\longrightarrow&\big(\pi(a_{ij})\big)_{m\times (2\nu+\delta)}
\end{eqnarray*}
and
\begin{eqnarray*}
\sigma_{(k_1,k_2,\ldots,k_\nu,\pi)}:V\big(Oi\big(2\nu,q\big)\big) &\longrightarrow& V\big(Oi\big(2\nu,q\big)\big)\\
A &\longrightarrow& \sigma_\pi(A)\cdot{\rm diag}(k_1,k_2,\ldots,k_{\nu}, k_1^{-1},k_2^{-1},\ldots,k_{\nu}^{-1}),
\end{eqnarray*}
where $\pi\in{\rm Aut}(\mathbb{F}_q)$, $k_1\in\mathbb{F}_q^{*2}$ and $k_2\ldots,k_{\nu}\in \mathbb{F}_q^*$.
It is easy to check that $\sigma_{\pi}\in {\rm Aut}\big(Oi\big(2\nu,q\big)\big)$. Moreover, since
${\rm diag}(k_1,k_2,\ldots,k_{\nu},k_1^{-1},k_2^{-1}\ldots,k_{\nu}^{-1}) S_{2\nu}\,{^t({\rm diag}(k_1,k_2,\ldots,k_{\nu}, k_1^{-1},k_2^{-1},\ldots,k_{\nu}^{-1}))}=S_{2\nu},$
we have ${\rm diag}(k_1,k_2,\ldots,k_{\nu},k_1^{-1},\ldots,k_{\nu}^{-1})\in O_{2\nu}(\mathbb{F}_q).$ So $\sigma_{(k_1,k_2,\ldots,k_{\nu},\pi)}\in {\rm Aut}\big(Oi\big(2\nu,q\big)\big).$

Before considering the problem of automorphism groups of $Oi\big(2\nu,q\big)$, let$'$s look at a simple fact.

\begin{Example} \label{e1} Let $\mathbb{F}_q$ be a field. Then $$V\big(Oi\big(2,q\big)\big)=\{[(1,x)]:x\in \mathbb{F}_q\}\bigcup \{[(0,1)]\}$$
and
$$E\big(Oi\big(2,q\big)\big)=\{[(1,x)]\text{---}[(1,y)]: xy=-1,x,y\in \mathbb{F}_q\}\bigcup\{[(1,0)]\text{---}[(1,0)], [(0,1)] \text{---}[(0,1)]\}.$$
Clearly, $Oi\big(2,q\big)$ is a disconnected graph.
\end{Example}


\begin{Theorem}\label{ot1}
Let $\nu\geq2$ and $E_{2\nu}$ be the subgroup of ${\rm Aut}\big(Oi\big(2\nu,q\big)\big)$ defined as follows:
$$E_{2\nu}=\big\{ \sigma\in {\rm Aut}\big(Oi\big(2\nu,q\big)\big): \sigma\big([e_{i}]\big)=[e_{i}]~and~\sigma\big([f_{i}]\big)=[f_{i}] ~\text{for}~1\leq i\leq\nu\big\}.$$
Then ${\rm Aut}\big(Oi\big(2\nu,q\big)\big)=PO_{2\nu}\big(\mathbb{F}_q\big)\cdot E_{2\nu}$. Moreover, let
\begin{eqnarray*}
h:\left.{\big((\mathbb{F}_q^{*2}\times\mathbb{F}_q^*\times\cdots\times \mathbb{F}_q^*)\rtimes_{\varphi_{2\nu}}{\rm Aut}(\mathbb{F}_q)\big)}\middle/ K_{2\nu}\right. &\longrightarrow& E_{2\nu}\\
(k_1,k_2,\ldots,k_\nu,\pi)K_{2\nu} &\longrightarrow& \sigma_{(k_1,k_2,\ldots,k_\nu,\pi)},
\end{eqnarray*}
where $1_{\mathbb{F}_q}$ is an identity element of ${\rm Aut}(\mathbb{F}_q)$ and
\begin{displaymath}
K_{2\nu}=\left\{ \begin{array}{ll}
\{(c,c,\ldots,c,1_{\mathbb{F}_q}):c=\pm1\}, &\ {\rm if}\ -1\in\mathbb{F}_q^{*2},\\
\{(1,1,\ldots,1,1_{\mathbb{F}_q}) \}, &\ {\rm if}\ -1\notin\mathbb{F}_q^{*2}.
\end{array} \right.
\end{displaymath}
Then $h$ is a group isomorphism from ~$\left.{\big((\mathbb{F}_q^{*2}\times\mathbb{F}_q^*\times\cdots\times \mathbb{F}_q^*)\rtimes_{\varphi_{2\nu}}{\rm Aut}(\mathbb{F}_q)\big)}\middle/ K_{2\nu}\right.$ to $E_{2\nu}$.
\end{Theorem}

\begin{proof}  Let $\tau\in {\rm Aut}\big(Oi\big(2\nu,q\big)\big)$. By Proposition \ref{2.3}, we can suppose that $\tau\big([e_{i}]\big)=[\alpha_{i}]$ and $\tau\big([f_{i}]\big) =[\beta_{i}]$ for $1\leq i\leq\nu$. Then
$$\alpha_i{S_{2\nu}}{^t\!\!\beta}_i\neq0,\alpha_iS_{2\nu}{^t\!\!\alpha}_j= \beta_iS_{2\nu}{^t\!\!\beta}_j=0\ {\rm for}\ 1\leq i,j\leq\nu$$
and
$$\alpha_iS_{2\nu}{^t\!\!\beta}_j=0\ {\rm for}\ i\neq j,\,1\leq i,j\leq\nu.$$
Let $A=[e_1,\ldots,e_{\nu},f_1,\ldots,f_{\nu}]$ and $A^{'}=[\alpha_{1}, \ldots, \alpha_{\nu},\beta_{1}, \ldots, \beta_{\nu}]$. Since $A^{'} S\, ^{t}\!A^{'}$ is a $2\nu \times 2\nu$ nonsingular symmetric matrix, we can choose $\alpha_i,\beta_i, 1\leq i\leq\nu,$ such that $\alpha_i S\, ^{t}\!\beta_i=1$. Then $AS\, ^{t}\!A=A^{'}S\, ^{t}\!A^{'}$. By \cite[Lemma 6.8]{zwG}, there exists $T\in O_{2\nu}(\mathbb{F}_q)$ such that $A=A^{'}T$. Let $\tau_1= \sigma_T\tau$. Then $\tau_1([e_i]) =[e_i]$ and $\tau_1([f_i]) =[f_i]$ for $1\leq i\leq\nu$. So $\tau_1\in E_{2\nu}$ and $\tau=\sigma_T^{-1}\tau_1\in PO_{2\nu}(\mathbb{F}_q) \cdot E_{2\nu}$. Thus ${\rm Aut}\big(Oi\big(2\nu,q\big)\big)= PO_{2\nu}(\mathbb{F}_q)\cdot E_{2\nu}$.
\vskip 0.15cm

Let $k_1\in\mathbb{F}_q^{*2},~k_2,\ldots,k_{\nu}\in\mathbb{F}_q^*$ and $\pi\in{\rm Aut}(\mathbb{F}_q)$. Then it is easy to check that
$$\sigma_{(k_1,k_2,\ldots,k_{\nu},\pi)}([e_i])=[e_i]\ {\rm and}\ \sigma_{(k_1,k_2,\ldots,k_{\nu},\pi)}([f_i])=[f_i]$$
for $1\leq i\leq\nu,$ and so $\sigma_{(k_1,k_2\ldots,k_{\nu},\pi)}\in E_{2\nu}$. Let
\begin{eqnarray*}
h^{'}:(\mathbb{F}_q^{*2}\times\mathbb{F}_q^*\times\cdots\times \mathbb{F}_q^*)\rtimes_{\varphi_{2\nu}} {\rm Aut}(\mathbb{F}_q) & \longrightarrow & E_{2\nu}\\
(k_1,k_2,\ldots,k_\nu,\pi) & \longrightarrow & \sigma_{(k_1,\ldots,k_\nu,\pi)}.
\end{eqnarray*}
Then by the definition of $\sigma_{(k_1,k_2,\ldots,k_\nu,\pi)}$ and Proposition\ref{2.4}, it is easily seen that $h^{'}$ is a group homomorphism and the kernel of $h^{'}$ is $K_{2\nu}$. Hence in order to show that $h$ is a group isomorphism from $\left.{\big((\mathbb{F}_q^{*2}\times\mathbb{F}_q^*\times\cdots\times \mathbb{F}_q^*)\rtimes_{\varphi_{2\nu}}{\rm Aut}(\mathbb{F}_q)\big)}\middle/ K_{2\nu}\right.$ to $E_{2\nu}$, it suffices to show that for any $\sigma\in E_{2\nu+1}$, there exist $k_1\in\mathbb{F}_q^{*2},$ $k_2,\ldots,k_{\nu}\in\mathbb{F}_q^*$ and $\pi\in{\rm Aut}(\mathbb{F}_q)$ such that $\sigma=\sigma_{(k_1,k_2,\ldots,k_\nu,\pi)}.$
\vskip 0.15cm

Let $\sigma\in E_{2\nu}$. And then we will prove that there exist ~$k_1\in\mathbb{F}_q^{*2},$ $k_2,\ldots,k_{\nu}\in\mathbb{F}_q^*$ and $\pi\in{\rm Aut}(\mathbb{F}_q)$ such that
$$\sigma=\sigma_{(k_1,k_2,\ldots,k_\nu,\pi)}.$$

Let $[(a_1,\ldots,a_{2\nu})]\in V\big(Oi\big(2\nu,q\big)\big)$ and suppose that $\sigma([(a_1,\ldots,a_{2\nu})])=[(b_1,\ldots,b_{2\nu})]$. Then
\begin{center}
$a_i=0$ if and only if $[(a_1,\ldots,a_{2\nu})]$---$[f_i]$
\end{center}
and
\begin{center}
 $a_{\nu+i}=0$ if and only if $[(a_1,\ldots,a_{2\nu})]$---$[e_i]$
\end{center}
for $1\leq i\leq\nu$. Since $\sigma([e_i])=[e_i]$ and $\sigma([f_i]) =[f_i]$, $1\leq i\leq\nu$, we have
$$a_k=0\text{ if and only if }b_k=0 \text{ for }1\leq k\leq2\nu.$$
Moreover, $[(a_1,\ldots,a_{2\nu})]\in V\big(Oi\big(2\nu,q\big)\big)$ can be written uniquely as $[(0,\ldots,0,1,a_{i+1}^{'},\ldots,a_{2\nu}^{'})]$ if $a_1=\cdots=a_{i-1}=0$ and $a_i\neq0$. Therefore $\sigma([(0,\ldots,0,1,a_{i+1}^{'},\ldots,a_{2\nu}^{'})])$ can be written uniquely as $[(0,\ldots,0,1,b_{i+1},\ldots,b_{2\nu})].$
\vskip 0.1cm

Let $1\leq k\leq\nu$. For $[f_1+ae_k],[f_1+bf_k]\in V\big(Oi\big(2\nu,q\big)\big)$, there exist $a^{'},b^{'}\in\mathbb{F}_q$ such that $\sigma\big([e_1+ae_k]\big)=[e_1+ a^{'}e_k]$ and $\sigma\big([e_1+ bf_k]\big)=[e_1+b^{'}e_k]$. Thus we can define permutations $\pi_i$ of $\mathbb{F}_q$ with $\pi_i(0)=0$ such that
$$\sigma\big([e_1+ a_ie_i]\big)=[e_1+\pi_i(a_i)e_i]\,\ {\rm and}\,\ \sigma\big([e_1+ a_{\nu+i}f_i]\big)=[e_1+\pi_{\nu+i}(a_{\nu+i})f_i],$$
where $2\leq i\leq\nu.$
\vskip 0.15cm

In the following proof, we need only consider two cases: $\nu=2$ and $\nu>2$.

(1) $\nu=2$. Clearly, $V\big(Oi\big(4,q\big)\big)$ consists of four types of vertices: $\{[(1,a_2,a_3,a_4)]: a_2,a_3,a_4\in \mathbb{F}_q\}, \{[(0,1,\linebreak[1]a_3,a_4)]: a_3,a_4\in \mathbb{F}_q\}, \{[(0,0,1,a_4)]: a_4\in \mathbb{F}_q\}$ and $\{[(0,0,0,1)]\}$. If $\sigma([(x_1,x_2,x_3,x_4)])=[(y_1,y_2,y_3,\linebreak[1]y_4)]$ and $x_i=0$, then we get $y_i=0$ for $1\leq i\leq4$. Thus without restriction of generality we can assume that $a_i\neq0$ for $1\leq i\leq4$. Now we first determine that $\sigma([(0,0,1,a_4)]),\sigma([(0,1,a_3,a_4)])$ and $\sigma([(1,a_2,a_3,a_4)])$ for $a_2,a_3,\linebreak[1]a_4\in \mathbb{F}_q$.

Suppose that $\sigma([(0,0,1,a_4)])=[(0,0,1,a_4^{'})]$. Since $[(0,0,1,a_4)] \text{---}[(1,-a_4^{-1},0,0)]$, $[(0,0,1,a_4^{'})] \text{---}\allowbreak[(1,\pi_2(-a_4^{-1}),0,0)]$, i.e., $a_4^{'}= -\pi_2(-a_4^{-1})^{-1}$.
Thus $\sigma([(0,0,1,a_4)])=[(0,0,1,-\pi_2(-a_4^{-1})^{-1})]$.

Suppose that $\sigma([(0,1,a_3,a_4)])=[(0,1,a_3^{'},a_4^{'})]$. Since $[(0,1,a_3,a_4)]\text{---}[(1,0,0,-a_3)]$, it follows that $[(0,1,a_3^{'}, a_4^{'})]\text{---}[(1,0,0,\pi_4(-a_3))]$, i.e., $a_3^{'}=-\pi_4(-a_3)$. Since $[(0,1,a_3,a_4)]\text{---}[(1, -a_3a_4^{-1},0,\linebreak[1]0)]$, we see that $[(0,1,-\pi_4(-a_3), a_4^{'})]\text{---}[(1,\pi_2(-a_3a_4^{-1}), 0,0)]$, i.e., $a_4^{'}=\pi_2(-a_3a_4^{-1})^{-1}\pi_4(-a_3)$. Thus $\sigma([(0,1,a_3,a_4)])= [(0,1,-\pi_4(-a_3),\pi_2(-a_3a_4^{-1})^{-1}\pi_4(-a_3))].$

Similarly, we can prove that
$$\sigma([(1,a_2,\linebreak[1] a_3,a_4)])= [(1,\pi_2(a_2),-\pi_2(-a_3a_4^{-1})\pi_4(a_4), \pi_4(a_4))].$$


Next we will prove that $\pi$ is an automorphism of $\mathbb{F}_q$ for $\pi=\pi_2(1)^{-1}\pi_2=\pi_4(1)^{-1}\pi_4$. Let $a,b\in \mathbb{F}_q$.

Since $[(1,a,0,a)]\text{---}[(1,1,0,-1)]$ and $[(1,a,0,a)] \text{---}[(1,-1,0,1)]$, we have $[(1,\pi_2(a), 0,\pi_4(a))] \text{---}\linebreak[1] [(1,\pi_2(1),0,\pi_4(-1))]$ and $[(1,\pi_2(a),0,\pi_4(a))] \text{---}[(1,\pi_2(-1), 0,\pi_4(1))]$, i.e., $\pi_2(a)\pi_4(-1)+ \pi_2(1)\pi_4(a)=0$ and $\pi_2(a)\pi_4(1)+\pi_2(-1)\pi_4(a)=0$. Let $a=1$, then $-\pi_2(1)=\pi_2(-1)$ and $-\pi_4(1)=\pi_4(-1)$. Thus we deduce that $\pi_2(1)^{-1}\pi_2= \pi_4(1)^{-1}\pi_4$. Set $\pi=\pi_2(1)^{-1}\pi_2=\pi_4(1)^{-1}\pi_4$.

Since $[(1,-a,0,a)]\text{---}[(1,1,0,1)]$, we have $[(1,\pi_2(-a),0,\pi_4(a))] \text{---}[(1,\pi_2(1),0, \linebreak[1]\pi_4(1))]$, i.e., $\pi_2(-a)\pi_4(1)+\pi_2(1)\pi_4(a)=0$. It follows that $\pi_2(-a)=-\pi_2(a)$ since $\pi_2(1)^{-1}\pi_2(-a)= -\pi_4(1)^{-1}\pi_4(a)=-\pi_2(1)^{-1}\pi_2(a)$. Similarly, we have $\pi_4(-a)=-\pi_4(a)$.

Since $[(1,ab,0,a)]\text{---}[(1,b,0,-1)]$, we have $[(1,\pi_2(ab), 0,\pi_4(a))]\text{---} [(1,\pi_2(b),0,-\pi_4(1))]$, i.e., $-\pi_2(ab)\pi_4(1)+\pi_4(a)\pi_2(b)=0$. Then
$\pi_2(ab)=\pi_4(1)^{-1}\pi_4(a)\pi_2(b)=\pi_2(1)^{-1}\pi_2(a)\pi_2(b).$ Thus
$$\pi(ab)=\pi_2(1)^{-1}\pi_2(ab)=\pi_2(1)^{-1}\pi_2(1)^{-1}\pi_2(a)\pi_2(b)= \pi(a)\pi(b).$$

Let $ab\neq0$. Since $[(1,a+b,0,1)]\text{---}[(0,1,ab^{-1}, -b^{-1})]$, we have $[(1,\pi_2(a+b),0, \pi_4(1))] \text{---} [(0,1,\linebreak[1]\pi_4(ab^{-1}), -\pi_4(ab^{-1})\pi_2(a)^{-1})]$, i.e., $\pi_4(ab^{-1})- \pi_2(a+b)\pi_4(ab^{-1})\pi_2(a)^{-1}+ \pi_4(1)=0$. Then
$$\pi_2(a+b)=\pi_2(a)+ \pi_2(a)\pi_4(1)\pi_4(ab^{-1})^{-1}= \pi_2(a)+ \pi_2(1)\pi(a)\pi(ab^{-1})^{-1}~~$$
$$~\,=\pi_2(a)+\pi_2(1)\pi(a)\pi(a^{-1}b)=\pi_2(a)+\pi_2(b).~~~~~~~~~~~~\:~\:\,$$
Thus
$ \pi(a+b)=\pi_2(1)^{-1}\pi_2(a+b)=\pi_2(1)^{-1}(\pi_2(a)+\pi_2(b))= \pi(a)+\pi(b).$

Since $\pi_2$ is a permutation of $\mathbb{F}_q$, it follows that $\pi$ is injective. $\pi$ is surjective since $\mathbb{F}_q$ is finite. Thus $\pi$ is an automorphism of $\mathbb{F}_q$. Furthermore, we conclude that
$$\sigma([(a_1,a_2,a_3,a_4)])= [(\pi(a_1),\pi_2(1)\pi(a_2),\pi_2(1)\pi_4(1)\pi(a_3), \pi_4(1)\pi(a_4))].$$

Lastly, we will prove that there exist $k_1\in \mathbb{F}_q^{*2},k_2\in \mathbb{F}_q^*$ and $\pi\in {\rm Aut}(\mathbb{F}_q)$ such that $\sigma=\sigma_{(k_1,k_2,\pi)}$.

By Proposition \ref{2.5}, we know that $[(1,0,1,0)]$ and $\sigma([(1,0,1,0)])$ are the same type orthogonal subspace. Then we have $\pi_2(1)\pi_4(1)\in \mathbb{F}_q^{*2}$. Let $k_1=\pi_2(1)\pi_4(1),$ $k_2=\pi_2(1)$ and $\pi=\pi_2(1)^{-1}\pi_2=\pi_4(1)^{-1}\pi_4$. Then for $\big(a_{ij}\big)_{m\times4}\in V\big(Oi\big(4,q\big)\big)$, we have
\begin{displaymath}
\sigma\big(\big(a_{ij}\big)_{m\times4}\big)=\left(
\begin{array}{cccc}
k_1\pi(a_{11})&k_2\pi(a_{12}) &k_1^{-1}\pi(a_{13}) & k_2^{-1}\pi(a_{14})\\
\vdots&\vdots &\vdots&\vdots\\
k_1\pi(a_{m1})&k_2\pi(a_{m2}) &k_1^{-1}\pi(a_{m3}) & k_2^{-1}\pi(a_{m4})
\end{array}
\right).
\end{displaymath}
So $\sigma\big(\big(a_{ij}\big)_{m\times4}\big)= \sigma_{(k_1,k_2,\pi)}\big(\big(a_{ij}\big)_{m\times4}\big)$ and therefore $\sigma=\sigma_{(k_1,k_2,\pi)}.$
\vskip 1.5mm

(2) $\nu>2$. Similar to the case of $\nu=2$, we can prove that there exist $k_1\in\mathbb{F}_q^{*2},k_2,\ldots,k_{\nu}\in \mathbb{F}_q^*$ and $\pi\in{\rm Aut}(\mathbb{F}_q)$ such that $\sigma=\sigma_{(k_1,k_2,\ldots,k_{\nu},\pi)}.$
\end{proof}

\begin{Corollary}\label{3.3} Let $\nu\geq 1$ and $\mathbb{F}_q$ be a field of characteristic $p$. Then
\begin{displaymath}
\big|{\rm Aut}\big(Oi\big(2\nu,q\big)\big)\big|=\left\{ \begin{array}{ll}
2^{\frac{q+1}{2}}\cdot(\frac{q-1}{2})!, & {\rm\ if}\ \nu=1;\\
\frac{1}{2}q^{\nu(\nu-1)} \mathop{\prod}\limits_{i=1}^{\nu}(q^i-1) \mathop{\prod}\limits_{i=1}^{\nu-1}(q^i+1)[\mathbb{F}_q:\mathbb{F}_p],&\ {\rm if}\ \nu\geq2 \ {\rm and\ }-1\in\mathbb{F}_q^{*2};\\
q^{\nu(\nu-1)} \mathop{\prod}\limits_{i=1}^{\nu}(q^i-1) \mathop{\prod}\limits_{i=1}^{\nu-1}(q^i+1)[\mathbb{F}_q:\mathbb{F}_p],&\ {\rm if}\ \nu\geq2\ {\rm and\ }-1\notin\mathbb{F}_q^{*2}.
\end{array} \right.
\end{displaymath}
\end{Corollary}
\begin{proof} (1) Let $\nu=1$. Vertices $[e_1], [e_2]$ in $Oi\big(2,q\big)$ are totally isotropic subspaces. According to Example \ref{e1}, it is easy to calculate that $\big|{\rm Aut}\big(Oi\big(2,q\big)\big)\big|=2^{\frac{q+1}{2}}\cdot(\frac{q-1}{2})!$.

(2) Let $\nu\geq2$ and $-1\in\mathbb{F}_q^{*2}$. Then
$\big|{\rm Aut}\big(Oi\big(2\nu,q\big)\big)\big|= \frac{|PO_{2\nu}(\mathbb{F}_q)|\cdot|E_{2\nu}|}{|PO_{2\nu}(\mathbb{F}_q)\cap E{2\nu}|}.$ According to \cite[Theorem 6.21]{zwG}, we see that $|PO_{2\nu}(\mathbb{F}_q)|= \frac{1}{2}|O_{2\nu}(\mathbb{F}_q)|= \frac{1}{2}q^{\nu(\nu-1)} \prod_{i=1}^{\nu}(q^i-1) \prod_{i=0}^{\nu-1}(q^i+1).$ Clearly, we have $|E_{2\nu}|=\frac{1}{4}(q-1)^{\nu}\cdot[\mathbb{F}_q:\mathbb{F}_p]$, where $|{\rm Aut}(\mathbb{F}_q)|=[\mathbb{F}_q:\mathbb{F}_p]$. Let $\sigma_P\in PO_{2\nu}(\mathbb{F}_q)\cap E_{2\nu}$. Then $\sigma_P=\sigma_{(k_1, \ldots, k_{\nu},\pi)}$ for some $k_1\in \mathbb{F}_q^{*2},k_1,\ldots, k_{\nu}\in \mathbb{F}_q^*$ and $\pi\in {\rm Aut}(\mathbb{F}_q)$. So $|PO_{2\nu}(\mathbb{F}_q)\cap E_{2\nu}|=\frac{1}{4}(q-1)^{\nu}$ and thus
$$~~\big|{\rm Aut}\big(Oi\big(2\nu,q\big)\big)\big|= \frac{1}{2}q^{\nu(\nu-1)} \mathop{\prod}\limits_{i=1}^{\nu}(q^i-1) \mathop{\prod}\limits_{i=1}^{\nu-1}(q^i+1)[\mathbb{F}_q:\mathbb{F}_p].$$

(3) Let $\nu\geq2$ and $-1\notin\mathbb{F}_q^{*2}$. Similar to the case of $(2)$, we can prove that $\big|{\rm Aut}\big(Oi\big(2\nu,q\big)\big)\big|= q^{\nu(\nu-1)} \mathop{\prod}\limits_{i=1}^{\nu}(q^i-1) \mathop{\prod}\limits_{i=1}^{\nu-1}(q^i+1)[\mathbb{F}_q:\mathbb{F}_p].$
\end{proof}

Secondly, we will discuss the case of $\delta=1$. Since the groups $O_{2\nu+1,1}(\mathbb{F}_q)$ and $O_{2\nu+1,z}(\mathbb{F}_q)$ are isomorphic, we need only consider the case of $O_{2\nu+1,1}(\mathbb{F}_q)$. The other case is similar to prove.
Let $S=S_{2\nu+1,1}$. In what follows, we write $\varepsilon$ for $e_{2\nu+1}$. Then
$$e_i{S_{2\nu+1,1}}{^t\!f_i}=\varepsilon S_{2\nu+1,1}{^t\!\varepsilon}=1,e_i{S_{2\nu+1,1}}{^t\!e_j}=f_i{S_{2\nu+1,1}}{^t\!f_j}=0\ {\rm for}\ 1\leq i,j\leq\nu$$
and
$$e_i{S_{2\nu+1,1}}{^t\!f_j}=e_i{S_{2\nu+1,1}}{^t\!\varepsilon}= f_i{S_{2\nu+1,1}}{^t\!\varepsilon}=0\ {\rm for}\ i\neq j,1\leq i,j\leq\nu.$$

Similar to the case of $\delta=0$, let $\varphi_{2\nu+1}$ be the natural action of ${\rm Aut}(\mathbb{F}_q)$ on the group $\mathbb{F}_q^{*2}\times \mathbb{F}_q^*\times \cdots\times \mathbb{F}_q^*\times \mathbb{F}_3^*$ defined by $\varphi_{2\nu+1}(\pi)((k_1,k_2,\ldots, k_{\nu},\delta_1))= (\pi(k_1),\pi(k_2), \ldots, \pi(k_{\nu}),\pi(\delta_1))$ and $\big(\mathbb{F}_q^{*2}\times \mathbb{F}_q^*\times \cdots\times \mathbb{F}_q^*\times \mathbb{F}_3^*\big) \rtimes_{\varphi_{2\nu+1}} {\rm Aut}\big(\mathbb{F}_q\big)$ be a semidirect product group corresponding to $\varphi_{2\nu+1}$. Moreover, we define maps $\sigma_{(k_1,k_2,\ldots,k_{\nu},\delta_1,\pi)}$ from $V\big(Oi\big(2\nu+1,q\big)\big)$ to itself by
$$\sigma_{(k_1,k_2,\ldots,k_\nu, \delta_1,\pi)}(A)= \sigma_\pi(A)\cdot {\rm diag}(k_1,k_2,\ldots,k_\nu,k_1^{-1},k_2^{-1},\ldots,k_{\nu}^{-1},\delta_1),$$
where $k_1\in \mathbb{F}_q^{*2}$, $k_2,\ldots, k_{\nu}\in \mathbb{F}_q^*$, $\delta_1\in \mathbb{F}_3^*$ and $\pi\in {\rm Aut}(\mathbb{F}_q)$. It is easy to prove that $\sigma_{(k_1,k_2,\ldots,k_{\nu},\delta_1,\pi)}\in {\rm Aut}\big(Oi\big(2\nu+1,q\big)\big)$.

\begin{Theorem}\label{ot2}
Let $\nu\geq2$ and $E_{2\nu+1}$ be the subgroup of ${\rm Aut}\big(Oi\big(2\nu+1,q\big)\big)$ defined as follows
$$\big\{\sigma\in {\rm Aut}\big(Oi\big(2\nu+1,q\big)\big): \sigma\big([e_{i}]\big)=[e_{i}], \sigma\big([f_{i}]\big)=[f_{i}]\ {\rm and}\ \sigma\big([\varepsilon]\big)=[\varepsilon]\ {\rm for}\ 1\leq i\leq\nu\big\}.$$
Then ${\rm Aut}\big(Oi\big(2\nu+1,q\big)\big)=PO_{2\nu+1}(F_q)\cdot E_{2\nu+1}$. Moreover, let
\begin{eqnarray*}
h:\left.{\big((\mathbb{F}_q^{*2}\times\mathbb{F}_q^*\times\cdots\times \mathbb{F}_q^*\times \mathbb{F}_3^*)\rtimes_{\varphi_{2\nu+1}}{\rm Aut}(\mathbb{F}_q)\big)}\middle/ K_{2\nu+1}\right. &\longrightarrow& E_{2\nu+1}\\
(k_1,k_2,\ldots,k_\nu,\delta_1,\pi)K_{2\nu+1} &\longrightarrow& \sigma_{(k_1,k_2,\ldots,k_\nu,\delta_1,\pi)},
\end{eqnarray*}
where $1_{\mathbb{F}_q}$ is an identity element of ${\rm Aut}(\mathbb{F}_q)$, and
\begin{displaymath}
K_{2\nu+1}=\left\{ \begin{array}{ll}
\{(c,c,\ldots,c,c,1_{\mathbb{F}_q}):c=\pm1\}, &\ {\rm if}\ -1\in\mathbb{F}_q^{*2},\\
\{(1,1,\ldots,1,1,1_{\mathbb{F}_q}) \}, &\ {\rm if}\ -1\notin\mathbb{F}_q^{*2}.
\end{array} \right.
\end{displaymath}
Then $h$ is an isomorphism of groups from $\left.{\big((\mathbb{F}_q^{*2}\times\mathbb{F}_q^*\times\cdots\times \mathbb{F}_q^*\times \mathbb{F}_3^*)\rtimes_{\varphi_{2\nu+1}}{\rm Aut}(\mathbb{F}_q)\big)}\middle/ K_{2\nu+1}\right.$ to $E_{2\nu+1}$.
\end{Theorem}

\begin{proof} Similar to Theorem \ref{ot1}, we can prove that $${\rm Aut}\big(Oi\big(2\nu+1, q\big)\big)= PO_{2\nu+1,1}(\mathbb{F}_q)\cdot E_{2\nu+1}.$$
Define a map:
\begin{eqnarray*}
h^{'}:(\mathbb{F}_q^{*2}\times\mathbb{F}_q^*\times\cdots\times \mathbb{F}_q^*\times \mathbb{F}_3^*)\rtimes_{\varphi_{2\nu+1}} {\rm Aut}(\mathbb{F}_q) & \longrightarrow & E_{2\nu+1}\\
(k_1,k_2,\ldots,k_\nu,\delta_1,\pi) & \longrightarrow & \sigma_{(k_1,k_2,\ldots,k_\nu,\delta_1,\pi)}.
\end{eqnarray*}
By the definition of $\sigma_{(k_1,k_2,\ldots,k_\nu,\delta_1,\pi)}$, it is easy to prove  that $h^{'}$ is a group homomorphism and the kernel of $h^{'}$ is $K_{2\nu+1}$. Hence in order to show that $h$ is a group isomorphism, 
it suffices to show that for any $\sigma\in E_{2\nu+1}$, there exist $k_1\in\mathbb{F}_q^{*2},~k_2,\ldots,k_{\nu}\in\mathbb{F}_q^*$, $\delta_1\in \mathbb{F}_3^*$ and $\pi\in{\rm Aut}(\mathbb{F}_q)$ such that ~$$\sigma=\sigma_{(k_1,k_2,\ldots,k_{\nu},\delta_1,\pi)}.$$
\vskip 0.15cm

Let $\sigma\in E_{2\nu+1}$. And then we will prove that there exist $k_1\in\mathbb{F}_q^{*2}$, $k_2,\ldots,k_{\nu}\in\mathbb{F}_q^*$, $\delta_1\in \mathbb{F}_3^*$ and $\pi\in{\rm Aut}(\mathbb{F}_{q^2})$ such that $\sigma=\sigma_{(k_1,k_2,\ldots,k_{\nu},\delta_1,\pi)}$.

Let $[(a_1,a_2,\ldots,a_{2\nu+1})]\in V\big(Oi\big(2\nu+1,q\big)\big)$ and suppose that $\sigma([(a_1,a_2,\ldots,a_{2\nu+1})]) =[(b_1,b_2,\ldots,b_{2\nu+1})]$. Similar to the proof of the case $\delta=0$, we can conclude that
$$a_i=0\ {\rm if\ and\ only\ if}\ b_i=0,\ 1\leq i\leq2\nu+1,$$
and define permutations $\pi_j$ of $\mathbb{F}_q$ with $\pi_j(0)=0$ such that
$$\sigma\big([e_1+a_ke_k]\big)=[k_1e_1+\pi_k(a_k)e_k], \ $$
$$\ \ \sigma\big([e_1+a_{\nu+k}f_{k}]\big)= [k_1e_1+\pi_{\nu+k}(a_{\nu+k})f_{k}]\,$$
and
$$\ \ \ \ \ \sigma\big([e_1+a_{2\nu+1}\varepsilon]\big)= [k_1e_1+\pi_{2\nu+1}(a_{2\nu+1})\varepsilon],\ $$
where $j=2,\ldots,\nu,\nu+2,\ldots,2\nu+1$, $k=2,\ldots,\nu$ and $k_1^2=\pi_2(1)^{-1}\pi_4(1)^{-1}$.
\vskip 0.15cm

In the following proof, we need only consider two cases: $\nu=2$ and $\nu>2$.

(1) $\nu=2$. By Theorem \ref{ot1}, we know that $\sigma$ carries the vertex $[(a_1,a_2,a_3,a_4,a_{5})]$ of $Oi\big(5,q\big)$ into the vertex
$$[(k_1\pi(a_1),k_2\pi(a_2), k_1^{-1}\pi(a_{3}),k_2^{-1}\pi(a_{4}), x_{5})],$$
where $k_1=\pi_2(1)\pi_4(1),$ $k_2=\pi_2(1)$, $\pi=\pi_2(1)^{-1}\pi_2= \pi_4(1)^{-1}\pi_4$ and $x_{5}$ is to be determined. 
In order to determine the value of $x_{5},$ we first determine the images of $[e_i+a_{5}e_5]$ and $[f_i+a_{5}e_5]$ for $i=1,2$ under $\sigma$. Suppose that $\sigma([e_i+a_5\varepsilon])= [k_ie_i+a_{5}^{'}\varepsilon]$ and $\sigma([f_i+a_5\varepsilon])= [k_i^{-1}f_i+a_{5}^{'}\varepsilon]$ for $i=1,2$.

Suppose that $\sigma([(0,0,1,-a_{5},1)])= [(0,0,k_1^{-1},k_2^{-1}\pi(a_{5}),x)]$ and $\sigma([(0,0,1,-a_{5},-1)])= [(0,0,\allowbreak k_1^{-1},-k_2^{-1}\pi(a_{5}),y)]$.
Since $[(0,0,1,-a_{5}, 1)]$---$[(1,0,0,0,-1)]$, we have $[(0,0,k_1^{-1},-k_2^{-1}\pi(a_{5}),x)]$ --- $[(k_1,0,0,0,\pi_5(-1))],$ i.e., $x=-\pi_{5}(-1)^{-1}$. Since $[(0,1,0,0, a_{5})]$ --- $[(0,0,1,-a_5,1)] $, we have $[(0, k_2, 0,0,a_{5}^{'})]$ --- $[(0,0,k_1^{-1},-k_2^{-1}\pi(a_{5}), -\pi_{5}(-1)^{-1})]$, i.e., $a_{5}^{'}= -\pi_{5}(-1)\pi(a_{5})$.  Since $[(0,0,1,-a_{5},-1)]$ --- $[(1,0,0,0,1)]$, we have $[(0,0,k_1^{-1},-k_2^{-1}\pi(a_{5}),y)]$ --- $[(k_1,0,0,0,\pi_5(1))]$, i.e., $y=-\pi_{5}(1)^{-1}$. Since $[(0,1,0,0,- a_{5})]$ --- $[(0,0,1,-a_{5},-1)]$, we have $[(0,k_2,0,0,\pi_{5}(-1)\pi(a_{5}))]$ --- $[(0,0,k_1^{-1},-k_2^{-1}\pi(a_{5}),-\pi_{5}(1)^{-1})]$, i.e., $\pi_{5}(1)=-\pi_{5}(-1)$. Thus $\sigma([(0,1,0,0,a_{5})])= [(0,k_2,0,0,\pi_{5}(1)\pi(a_{5}))].$

Similarly, we can prove that $\sigma([(0,0,0,1,a_{5})])= [(0,0,0,k_2^{-1},\pi_{5}(1)\pi(a_{5}))]$, $\sigma([(1,0,0,0,a_{5})])= [(k_1,0,0,0,\pi_{5}(1)\pi(a_{5}))]$ and $\sigma([(0,0,1,0,a_{5})])= [(0,0,k_1^{-1},0,\pi_{5}(1)\pi(a_{5}))]$.




Since $[(1,0,0,0,1)]$ --- $[(0,0,1,0,-1)]$, it follows that $[(k_1,0,0,0,\pi_{5}(1))]$ --- $[(0,0,k_1^{-1},0,\pi_{5}(1))]$, i.e., $\pi_{5}(1)^2=1$. Hence we have $\pi_{5}(1)\in F_3^*$. It is easy to check that
$$\sigma([(a_1,a_2,a_3,a_4,,a_{5})])=[(k_1\pi(a_1),k_2\pi(a_2), k_1^{-1}\pi(a_{3}),k_2^{-1}\pi(a_{4}), \pi_5(1)\pi(a_{5}))].$$

Lastly, Let $k_1=\pi_2(1)\pi_{4}(1), k_2=\pi_{2}(1), \delta_1=\pi_5(1)$ and $\pi=\pi_2(1)^{-1}\pi_2=\pi_4(1)^{-1}\pi_4=\pi_5(1)^{-1}\pi_5$. Then for $A=\big(a_{ij}\big)_{m\times5}\in V\big(Oi\big(5,q\big)\big)$,
\begin{displaymath}
\sigma\big(A\big)=\left(
\begin{array}{ccccc}
k_1\pi(a_{11}) & k_2\pi(a_{12}) &k_1^{-1}\pi(a_{13}) &k_2^{-1}\pi(a_{14}) & \delta_1\pi(a_{15})\\
\vdots&\vdots & \vdots& \vdots& \vdots\\
k_1\pi(a_{m1}) & k_2\pi(a_{m2}) &k_1^{-1}\pi(a_{m3}) & k_2^{-1}\pi(a_{m4})&\delta_1\pi(a_{15})
  \end{array}
\right)=\sigma_{(k_1,k_2,\delta_1,\pi)}\big(A\big).
\end{displaymath}
So $\sigma=\sigma_{(k_1,k_2,\delta_1,\pi)}.$
\vskip 1.5mm

(2) $\nu>2$. Similar to the case of $\nu=2$, we can prove that there exist $k_1\in\mathbb{F}_q^{*2},k_2\ldots,k_{\nu}\in \mathbb{F}_q^*, \delta_1\in\mathbb{F}_3^{*2}$ and $\pi\in{\rm Aut}(\mathbb{F}_q)$ such that $\sigma=\sigma_{(k_1,k_2\ldots,k_{\nu},\delta_1,\pi)}.$
\end{proof}

\begin{Corollary} Let $\nu\geq2$ and $\mathbb{F}_q$ be a field of characteristic $p$. Then
\begin{displaymath}
\big|{\rm Aut}\big(Oi\big(2\nu+1,q\big)\big)\big|=\left\{ \begin{array}{ll}
\frac{1}{2}q^{\nu^2} \mathop{\prod}\limits_{i=1}^{\nu}(q^i-1) \mathop{\prod}\limits_{i=1}^{\nu}(q^i+1)[\mathbb{F}_q:\mathbb{F}_p], & \textrm{µ±}~-1\in\mathbb{F}_q^{*2},\\
q^{\nu^2} \mathop{\prod}\limits_{i=1}^{\nu}(q^i-1) \mathop{\prod}\limits_{i=1}^{\nu}(q^i+1)[\mathbb{F}_q:\mathbb{F}_p], & \textrm{µ±}~-1\notin\mathbb{F}_q^{*2},
\end{array} \right.
\end{displaymath}

\end{Corollary}
\begin{proof}
The proof is similar to that of Corollary \ref{3.3}.
\end{proof}

Finally, we will discuss the case of $\delta=2$. Let $S=S_{2\nu+2,\Delta}$ and $\Delta={\rm diag}(1,-z)$ with $-1\in\mathbb{F}_q^{*2}$. In what follows, we write $\kappa$ for $e_{2\nu+2}$. Then
\begin{center}
$e_i{S_{2\nu+2}}{^t\!f_i}=\varepsilon S_{2\nu+2}{^t\!\varepsilon}=1, \kappa S_{2\nu+2}{^t\!\kappa}=-z,e_i{S_{2\nu+2}}{^t\!e_j}=f_i{S_{2\nu+2}}{^t\!f_j}=0$ for $1\leq i,j\leq\nu$
\end{center}
and
\begin{center}
$e_i{S_{2\nu+2}}{^t\!f_j}=e_i{S_{2\nu+2}}{^t\!\varepsilon}= f_i{S_{2\nu+2}}{^t\!\varepsilon}=e_i{S_{2\nu+2}}{^t\!\kappa}= f_i{S_{2\nu+2}}{^t\!\kappa}=0$ for $i\neq j,1\leq i,j\leq\nu.$
\end{center}

Similar to the case of $\delta=0$, let $\varphi_{2\nu+2}$ be the natural action of ${\rm Aut}(\mathbb{F}_q)$ on the group $\mathbb{F}_q^{*2}\times\mathbb{F}_q^*\times\cdots\times\mathbb{F}_q^* \times\mathbb{F}_3^*\times\mathbb{F}_3^*$ defined by $\varphi_{2\nu+2}(\pi)((k_1,k_2,\ldots, k_{\nu},\delta_1,\delta_2))= (\pi(k_1),\pi(k_2),\ldots,\pi(k_{\nu}),\pi(\delta_1),\pi(\delta_2))$ and $\big(\mathbb{F}_q^{*2}\times \mathbb{F}_q^*\times \cdots\times \mathbb{F}_q^*\times \mathbb{F}_3^*\times \mathbb{F}_3^*\big) \rtimes_{\varphi_{2\nu+2}} {\rm Aut}\big(\mathbb{F}_q\big)$ be a semidirect product group corresponding to $\varphi_{2\nu+2}$. Moreover, we define maps $\sigma_{(k_1,k_2,\ldots,k_{\nu},\delta_1,\delta_2,\pi)}$ from $V\big(Oi\big(2\nu+2,q\big)\big)$ to itself by
$$\sigma_{(k_1,k_2,\ldots,k_\nu,\delta_1,\delta_2,\pi)}(A)= \sigma_\pi(A)\cdot {\rm diag}(k_1,k_2,\ldots,k_\nu,k_1^{-1},k_2^{-1},\ldots,k_{\nu}^{-1},\delta_1,\delta_2),$$
where $k_1\in \mathbb{F}_q^{*2},k_2,\ldots,k_{\nu}\in \mathbb{F}_q^*,\delta_1,\delta_2\in \mathbb{F}_3^*$ and $\pi\in {\rm Aut}(\mathbb{F}_q)$. It is easy to prove that $\sigma_{(k_1,k_2,\ldots,k_{\nu},\delta_1,\delta_2,\pi)}\in {\rm Aut}\big(Oi\big(2\nu+2,q\big)\big)$.




\begin{Theorem}  \label{ot3}
Let $\nu\geq2$ and $E_{2\nu+2}$ be the subgroup of ${\rm Aut}\big(Oi\big(2\nu+2,q\big)\big)$ defined as follows:
$$\big\{\sigma\in {\rm Aut}\big(Oi\big(2\nu+2,q\big)\big): \sigma([e_{i}])=[e_{i}], \sigma([f_{i}])=[f_{i}],\sigma([\varepsilon])=[\varepsilon] \ {\rm and }\ \sigma([\kappa])=[\kappa]\ {\rm for}\ 1\leq i\leq\nu\big\}.$$
Then ${\rm Aut}\big(Oi\big(2\nu+2,q\big)\big)=PO_{2\nu+2}\big(\mathbb{F}_q\big)\cdot E_{2\nu+2}$. Moreover, let
\begin{eqnarray*}
h:\left.{\big((\mathbb{F}_q^{*2}\times\mathbb{F}_q^*\times\cdots\times \mathbb{F}_q^*\times \mathbb{F}_3^*\times \mathbb{F}_3^*)\rtimes_{\varphi_{2\nu+2}}{\rm Aut}(\mathbb{F}_q)\big)}\middle/ K_{2\nu+2}\right. &\longrightarrow& E_{2\nu+2}\\
(k_1,k_2,\ldots,k_\nu,\delta,\pi)K_{2\nu+2} &\longrightarrow& \sigma_{(k_1,k_2,\ldots,k_\nu,\delta_1,\delta_2,\pi)},
\end{eqnarray*}
where $1_{\mathbb{F}_q}$ is an identity element of ${\rm Aut}(\mathbb{F}_q)$, and
$K_{2\nu+2}=\{(c,c,\ldots,c,c,c,1_{\mathbb{F}_q}) \in(\mathbb{F}_q^{*2}\times\mathbb{F}_q^*\times\cdots\times \mathbb{F}_q^* \times \mathbb{F}_3^* \times \mathbb{F}_3^*)\rtimes_{\varphi_{2\nu+2}}{\rm Aut}(\mathbb{F}_q):c=\pm1\}$.
Then $h$ is an isomorphism of groups.
\end{Theorem}

\begin{proof} Similar to Theorem \ref{ot1}, we can show that ${\rm Aut}\big(Oi\big(2\nu+2, q\big)\big)=PO_{2\nu+2}\big(\mathbb{F}_q\big)\cdot E_{2\nu+2}$.
Define a map:
\begin{eqnarray*}
h^{'}:(\mathbb{F}_q^{*2}\times\mathbb{F}_q^*\times\cdots\times \mathbb{F}_q^*\times \mathbb{F}_3^*\times \mathbb{F}_3^*)\rtimes_{\varphi_{2\nu+2}} {\rm Aut}(\mathbb{F}_q) & \longrightarrow & E_{2\nu+2}\\
(k_1,k_2,\ldots,k_\nu,\delta_1,\delta_2,\pi) & \longrightarrow & \sigma_{(k_1,k_2,\ldots,k_\nu,\delta_1,\delta_2,\pi)}.
\end{eqnarray*}
By the definition of $\sigma_{(k_1,k_2,\ldots,k_\nu,\delta_1,\delta_2,\pi)}$, it is easy to prove  that $h^{'}$ is a group homomorphism and the kernel of $h^{'}$ is $K_{2\nu+2}$. Hence in order to show that $h$ is a group isomorphism,
it suffices to show that for any $\sigma\in E_{2\nu+2}$, there exist $k_1\in\mathbb{F}_q^{*2},k_2,\ldots,k_{\nu}\in\mathbb{F}_q^*$, $\delta_1,\delta_2\in \mathbb{F}_3^*$ and $\pi\in{\rm Aut}(\mathbb{F}_q)$ such that ~$$\sigma=\sigma_{(k_1,k_2,\ldots,k_{\nu},\delta_1,\delta_2,\pi)}.$$

Let $\sigma\in E_{2\nu+2}$. And then we will prove that there exist $k_1\in\mathbb{F}_q^{*2}$, $k_2,\ldots,k_{\nu}\in\mathbb{F}_q^*$, $\delta_1,\delta_2\in \mathbb{F}_3^*$ and $\pi\in{\rm Aut}(\mathbb{F}_{q^2})$ such that $\sigma=\sigma_{(k_1,k_2,\ldots,k_{\nu},\delta_1,\delta_2,\pi)}$.

Let $[(a_1,a_2,\ldots,a_{2\nu+1},a_{2\nu+2})]\in V\big(Oi\big(2\nu+2,q\big)\big)$ and suppose that $\sigma([(a_1,a_2,\ldots,a_{2\nu+1},a_{2\nu+2})]) =[(b_1,b_2,\ldots,b_{2\nu+1},b_{2\nu+2})]$. Similar to the proof of the case $\delta=0$ and $\delta=1$, we can conclude that
$$a_i=0\ {\rm if\ and\ only\ if}\ b_i=0\ {\rm for}\ 1\leq i\leq2\nu+2,$$
and define permutations $\pi_j$ of $\mathbb{F}_q$ with $\pi_j(0)=0$ such that
$\sigma\big([e_1+a_ke_k]\big)=[k_1e_1+\pi_k(a_k)e_k],$ $\sigma\big([e_1+a_{\nu+k}f_{k}]\big)= [k_1e_1+\pi_{\nu+k}(a_{\nu+k})f_{k}],$
$\sigma\big([e_1+a_{2\nu+1}\varepsilon]\big)= [k_1e_1+\pi_{2\nu+1}(a_{2\nu+1})\varepsilon]$
and $\sigma\big([e_1+a_{2\nu+2}\kappa]\big)= [k_1e_1+\pi_{2\nu+2}(a_{2\nu+2})\kappa],$
where $k=2,\ldots,\nu$ and $k_1^{2}=\pi_2(1)^{-1}\pi_4(1)^{-1}$.

\vskip 0.15cm

In the following proof, we need only consider two cases: $\nu=2$ and $\nu>2$.

(1) $\nu=2$. By Theorem \ref{ot2}, we know that $\sigma$ carries the vertex $[(a_1,a_2,a_3,a_4,a_5,a_6)]$ of $Oi\big(6,q\big)$ into the vertex
$$[(k_1\pi(a_1),k_2\pi(a_2), k_1^{-1}\pi(a_3),k_2^{-1}\pi(a_{4}),\delta_1\pi(a_5), x_6)],$$
where $k_1^{-1}k_2=\pi_2(1)$, $k_1^2=\pi_2(1)^{-1}\pi_4(1)^{-1}$, $\delta_1=\pi_5(1)$, $\pi=\pi_2(1)^{-1}\pi_2=\pi_4(1)^{-1}\pi_4= \pi_5(1)^{-1}\pi_5$ and $x_6$ is to be determined. 
In order to determine the value of $x_6,$ we first determine the images of ~$[e_i+a_6\kappa]$, $[f_i+a_6\kappa]$ and $[\varepsilon+a_6\kappa]$, $i=1,2$ under $\sigma$. Suppose that $\sigma([e_i+a_6\kappa])= [k_ie_i+a_6^{'}\kappa]$, $\sigma([f_i+a_6\kappa])= [k_i^{-1}f_i+a_6^{'}\kappa]$, $i=1,2$, and $\sigma([\varepsilon+a_6\kappa])= [\delta_1\varepsilon+a_6^{'}\kappa]$.

Suppose that $\sigma([e_3+a_6e_4+z^{-1}\kappa])= [k_1e_3+k_2^{-1}\pi(a_6)e_4+x\kappa]$.
Since $[e_3+a_6e_4+z^{-1}\kappa]$--- $[e_{1}+\kappa]$, we have $[k_1^{-1}e_3+k_2^{-1}\pi(a_6)e_4+x\kappa]$---$[k_1e_1+\pi_6(1)\kappa]$, i.e., $x=z^{-1}\pi_6(1)^{-1}$. Since $[e_2+a_6\kappa]$ --- $[e_3+ a_6e_4+
z^{-1}\kappa]$, we have $[k_2e_2+a_6^{'}\kappa]$ --- $[k_1^{-1}e_3+k_2^{-1}\pi(a_6)e_4 +z^{-1}\pi_6(1)^{-1}\kappa]$, i.e., $a_6^{'}= \pi_6(1)\pi(a_6)$. Thus $\sigma([e_2+a_6\kappa])=[k_2e_2+\pi_6(1)\kappa]$.

Similarly, we can prove that $\sigma([e_4+a_{6}\kappa])= [k_2^{-1}e_4+\pi_6(1)\pi(a_{6})\kappa],$ $\sigma([e_{1}+ a_{6}\kappa])=[k_1e_{1}+ \pi_6(1)\pi(a_6)\kappa],$ $\sigma([e_3+ a_{6}\kappa])= [k_1^{-1}e_3+\pi_6(1)\pi(a_{6})\kappa]$ and $\sigma([\varepsilon+a_{6}\kappa])= [\delta_1\varepsilon+\pi_6(1)\pi(a_{6})\kappa].$


We know that $[e_2+\kappa]$---$[e_4+z^{-1}\kappa]$, and then $[k_2e_2+\pi_6(1)\kappa]$---$[k_2^{-1}e_4+\pi(z^{-1})\pi_6(1)\kappa]$. Thus there exists $\delta\in \mathbb{F}_3^*$ such that $\pi_6(1)=\delta\sqrt{\pi(z)z^{-1}}.$ Moreover, it is easy to check that
$$\sigma([(a_1,a_2,a_3,a_4,,a_5,a_6)])=
[(k_1\pi(a_1),k_2\pi(a_2), k_1^{-1}\pi(a_3),k_2^{-1}\pi(a_{4}),\delta_1\pi(a_5), \pi_6(1)\pi(a_6))].$$

Lastly, Let $k_1^{-1}k_2=\pi_2(1)$, $k_1=\pi_2(1)^{-1}\pi_4(1)^{-1}$, $\delta_1=\pi_5(1)$, $\delta_2=\delta$ and $\pi=\pi_2(1)^{-1}\pi_2=\pi_4(1)^{-1}\pi_4= \pi_5(1)^{-1}\pi_5=\pi_6(1)^{-1}\pi_6$. Then for $A=\big(a_{ij}\big)_{m\times6}\in V\big(Oi\big(6,q\big)\big)$,
\begin{displaymath}
\sigma\big(A\big)=\left(
\begin{array}{cccccc}
k_1\pi(a_{11}) & k_2\pi(a_{12}) &k_1^{-1}\pi(a_{13}) &k_2^{-1}\pi(a_{14}) & \delta_1\pi(a_{15})& \delta_2\sqrt{\pi(z)z^{-1}}\pi(a_{16}) )\\
\vdots&\vdots & \vdots& \vdots& \vdots & \vdots\\
k_1\pi(a_{m1}) & k_2\pi(a_{m2}) &k_1^{-1}\pi(a_{m3}) & k_2^{-1}\pi(a_{m4})&\delta_1\pi(a_{15})& \delta_2\sqrt{\pi(z)z^{-1}}\pi(a_{m6}) )
  \end{array}
\right)
\end{displaymath}
\begin{displaymath}
=\,\sigma_{(k_1,k_2,\delta_1,\delta_2,\pi)}\big(A\big).\ \ \ \ \ \ \ \ \ \ \ \ \ \ \ \ \ \ \ \ \ \ \ \ \ \ \ \ \ \ \ \ \ \ \ \ \ \ \ \ \ \ \ \ \ \ \ \ \ \ \ \ \ \ \ \ \ \ \ \ \ \ \ \ \ \ \ \ \ \ \ \ \ \ \ \ \ \ \ \ \ \
\end{displaymath}
So $\sigma=\sigma_{(k_1,k_2,\delta_1,\delta_2,\pi)}.$
\vskip 1.5mm

(2) $\nu>2$. Similar to the case of $\nu=2$, we can prove that there exist $k_1\in\mathbb{F}_q^{*2},k_2\ldots,k_{\nu}\in \mathbb{F}_q^*, \delta_1,\delta_2\in\mathbb{F}_3^{*2}$ and $\pi\in{\rm Aut}(\mathbb{F}_q)$ such that $\sigma=\sigma_{(k_1,k_2\ldots,k_{\nu},\delta_1,\delta_2,\pi)}.$
\end{proof}

\begin{Corollary} Let $\nu\geq2$ and $F_q$ be a field of characteristic $p$. Then
\begin{displaymath}
\big|{\rm Aut}\big(Oi\big(2\nu+2,q\big)\big)\big|=\frac{1}{2}q^{\nu(\nu+1)} \mathop{\prod}\limits_{i=1}^{\nu}(q^i-1) \mathop{\prod}\limits_{i=1}^{\nu+1}(q^i+1)[\mathbb{F}_q:\mathbb{F}_p].
\end{displaymath}
\end{Corollary}

\begin{proof}
The proof is similar to that of Corollary \ref{3.3}.
\end{proof}

\section{The action of Automorphism groups of orthogonal inner graphs of odd characteristic}

\begin{Theorem}\label{4.1} Let $Oi\big(2\nu+\delta,q\big)$ be the orthogonal inner product graph over $\mathbb{F}_q$ and ~$\mathcal{M}(m,2s+\gamma,s,\Gamma;2\nu+\delta,\Delta)\neq\varnothing$ for $1\leq m<2\nu+\delta$. Then $\mathcal{M}(m,2s+\gamma,s,\Gamma;2\nu+\delta,\Delta)$ is exactly one orbit of $V\big(Oi\big(2\nu+\delta,q\big)\big)$ under the action of  ${\rm Aut}\big(Oi\big(2\nu+\delta,q\big)\big)$.
\end{Theorem}

\begin{proof}
By \cite[Lemma 6.4]{zwG} and Proposition \ref{2.4}, we know that for any ~$A,B\in\mathcal{M}(m,2s+\gamma,s,\Gamma;2\nu+\delta,\Delta)$, there exists ~$T\in\mathcal{F}_{2\nu+\delta,\Delta}(q)$ such that $\sigma_T(A)=\sigma_T(B).$
Thus ${\rm Aut}\big(Oi\big(2\nu+\delta,q\big)\big)$ is transitive on $\mathcal{M}(m,2s+\gamma,s,\Gamma;2\nu+\delta,\Delta)$.

In order to prove that $\mathcal{M}(m,2s+\gamma,s,\Gamma;2\nu+\delta,\Delta)$ is exactly one orbit of $V\big(Oi\big(2\nu+\delta,q\big)\big)$ under the action of  ${\rm Aut}\big(Oi\big(2\nu+\delta,q\big)\big)$, we need only need to show that  ~$\sigma\in{\rm Aut}\big(Oi\big(2\nu+\delta,q\big)\big)$, we have
$\sigma(\mathcal{M}(m,2s+\gamma,s,\Gamma;2\nu+\delta,\Delta))= \mathcal{M}(m,2s+\gamma,s,\Gamma;2\nu+\delta,\Delta).$ In the following proof, we will consider three cases: $\delta=0,1$ and $2$.

Let $\delta=0$, $\sigma\in{\rm Aut}\big(Oi\big(2\nu,q\big)\big)$ and ~$A\in\sigma(\mathcal{M}(m,2s+\gamma,s,\Gamma;2\nu,\Delta)$.

By Theorem \ref{ot1}, we know that there exist $k_1\in\mathbb{F}_q^{*2},~k_2,\ldots,k_{\nu}\in\mathbb{F}_q^*$ and $\pi\in{\rm Aut}(\mathbb{F}_q)$ such that $\sigma=\sigma_{(k_1,k_2,\ldots,k_{\nu},\pi)}.$
By the knowledge of  the finite field, we know that for non-square element $z\in\mathbb{F}_q$, there exists $a\in\mathbb{F}_q^{*2}$ such that $z=a\pi(z)$. If $AS\,{^t\!A}$ is cogredient to $M(m_1,2s_1+\gamma_1,s_1,\Gamma_1)$, then it is easy to check that $\sigma(A)S\,{^t\!(\sigma(A))}$ is cogredient to $M(m_1,2s_1+\gamma_1,s_1,\Gamma_1).$
So $A$ and $\sigma(A)$ are the same type orthogonal subspace. Thus we have
$\sigma(\mathcal{M}(m,2s+\gamma,s,\Gamma;2\nu,\Delta))= \mathcal{M}(m,2s+\gamma,s,\Gamma;2\nu,\Delta).$

Similar to the case of $\delta=0$, we can prove that
$\sigma(\mathcal{M}(m,2s+\gamma,s,\Gamma;2\nu+\delta,\Delta))= \mathcal{M}(m,2s+\gamma,s,\Gamma;2\nu+\delta,\Delta)$ when $\delta=1$ and $\delta=2$.

Thus $\mathcal{M}(m,2s+\gamma,s,\Gamma;2\nu+\delta,\Delta)$ is exactly one orbit of $V\big(Oi\big(2\nu+\delta,q\big)\big)$ under the action of  ${\rm Aut}\big(Oi\big(2\nu+\delta,q\big)\big)$.
\end{proof}

For $A\in Oi\big(2\nu+\delta,q\big)$, we define
$t(A)=(m,2s+\gamma,s,\Gamma;2\nu+\delta,\Delta)$ if the type of $A$ is $(m,2s+\gamma,s,\Gamma;2\nu+\delta,\Delta)$ and also define
\begin{displaymath}
\phi(A)= \left\{ \begin{array}{ll}
0,     & {\rm\ if\ |A|\ is\ a\ non\text{-square}\ element},\\
1,     &{\rm\ if\ |A|\ is\ a\ square\ element}.
\end{array} \right.
\end{displaymath}
For $a\in\mathbb{F}_q^*$, we set $a^0=1$ and $a^1=a$.

\begin{Proposition}\label{4.2} Let $Oi\big(2\nu+\delta,q\big)$ be the orthogonal inner product graph over $\mathbb{F}_q$ and $X_1\text{---}X_2,\linebreak[4]Y_1\text{---}Y_2\in E\big(Oi\big(2\nu+\delta,q\big)\big).$ Then $X_1$---$X_2$ and $Y_1$---$Y_2$ are in the same orbit of $E\big(Oi\big(2\nu+\delta,q\big)\big)$ under the action of $O_{2\nu+\delta,\Delta}\big(\mathbb{F}_q\big)$ if and only if one of the following is true$:$ $(1)$ $t(X_1)=t(Y_1),$ $t(X_2)=t(Y_2),$ $t(X_1+X_2)=t(Y_1+Y_2);$ $(2)$ $t(X_1)=t(Y_2),$ $t(X_2)=t(Y_1),$ $t(X_1+X_2)=t(Y_1+Y_2).$
\end{Proposition}

\begin{proof} ($\Rightarrow$) Suppose that $X_1$---$X_2$ and $Y_1$---$Y_2$ are in the same orbit of $E\big(Oi\big(2\nu+\delta,q\big)\big)$ under the action of $O_{2\nu+\delta,\Delta}(\mathbb{F}_q)$. Then there exists $T\in O_{2\nu+\delta,\Delta}(\mathbb{F}_q)$ such that one of the following is true: $(1) X_1T=Y_1,$ $X_2T=Y_2$; $(2)X_1T=Y_2,$ $X_2T=Y_1.$ Without loss of generality we can assume that $X_1T=Y_1$ and $X_2T=Y_2$. Then $t(X_1)=t(Y_1)$ and $t(X_2)=t(Y_2)$ by Proposition \ref{2.4} and Theorem \ref{4.1}. It is sufficient to prove $t(X_1+X_2)=t(Y_1+Y_2)$. Let $t(X_1)=t(Y_1)=(m_1,2s_1+\gamma_1,s_1,\Gamma_1;2\nu+\delta,\Delta)$ and $t(X_2)=t(Y_2)=(m_2,2s_2+\gamma_2,s_2,\Gamma_2;2\nu+\delta,\Delta)$. Then we have
\begin{displaymath}
  X_1T=\begin{bmatrix} X_{11} \\ \vdots \\ X_{1m_1}\end{bmatrix}T=
  \begin{bmatrix} Y_{11} \\ \vdots \\ Y_{1m_1}\end{bmatrix}=Y_1
  {\rm\ and\ }X_2T=\begin{bmatrix} X_{21} \\ \vdots \\ X_{2m_2}\end{bmatrix}T=
  \begin{bmatrix} Y_{21} \\ \vdots \\ Y_{2m_2}\end{bmatrix}=Y_2.
\end{displaymath}
It is easy to check that $(X_1+X_2)T=(Y_1+Y_2)$. By Lemma \cite[Lemma 6.4]{zwG}, we have $t(X_1+X_2)=t(Y_1+Y_2)$.
 \vskip 0.13cm

When $X_1T=Y_2$ and $X_2T=Y_1,$ we conclude similarly that $t(X_1)=t(Y_2),$ $t(X_2)=t(Y_1)$ and $t(X_1+X_2)=t(Y_1+Y_2).$
\vskip 0.15cm

($\Leftarrow$) Suppose that one of the following is true: $(1)$ $t(X_1)=t(Y_1),$ $t(X_2)=t(Y_2),$ $t(X_1+X_2)=t(Y_1+Y_2)$; $(2)$ $t(X_1)=t(Y_2),$ $t(X_2)=t(Y_1),$ $t(X_1+X_2)=t(Y_1+Y_2)$. Without loss of generality we can assume that $t(X_1)=t(Y_1)=(m_1,r_1),$ $t(X_2)=t(Y_2)=(m_2,r_2)$ and $t(X_1+X_2)=t(Y_1+Y_2)=(m,r).$ Then there exist matrix representations
\begin{displaymath}
  \begin{bmatrix} \alpha_{1} \\ \vdots \\ \alpha_{m-m_2} \\ \gamma_1 \\ \vdots \\ \gamma_{m_1+m_2-m}\end{bmatrix},~ \begin{bmatrix}\beta_{1} \\ \vdots \\ \beta_{m-m_1} \\ \gamma_1 \\ \vdots \\ \gamma_{m_1+m_2-m}\end{bmatrix}~\text{and}~ \begin{bmatrix} \gamma_1 \\ \vdots \\ \gamma_{m_1+m_2-m}\end{bmatrix}
\end{displaymath}
of $X_1,$ $X_2$ and $X_1\bigcap X_2$ respectively such that
\begin{displaymath}
 X_1+X_2=\begin{bmatrix} \alpha_{1} \\ \vdots \\ \alpha_{m-m_2} \\ \beta_{1}\\ \vdots \\ \beta_{m-m_1} \\ \gamma_1 \\ \vdots \\ \gamma_{m_1+m_2-m}\end{bmatrix},
\end{displaymath}
and
$$\alpha_iS\,{^t\!\beta_j}=0,\: \alpha_iS\,{^t\!\gamma_k}=0,\: \beta_jS\,{^t\!\gamma_k}=0$$
where $1\leq i\leq m-m_2$, $1\leq j\leq m-m_1$ and $1\leq k\leq m_1+m_2-m$. It is easy to verify that the type of $t(X_1\bigcap X_2)$ is determine by $X_1$ and $X_2$. We can choose a suitable basis
$$\alpha_{1},\ldots,\alpha_{m-m_2},~\beta_{1},\ldots,\beta_{m-m_1},~ \gamma_1,\ldots,\linebreak[4]\gamma_{m_1+m_2-m}$$
of $X_1+X_2$
such that $(X_1+X_2)S\,{^t(X_1+X_2)}=$
$${\rm diag}(x_1,\ldots,x_{r_1-1},z^{\phi(X_1)},0,\ldots,0,x_1,\ldots,x_{r_2-1},z^{\phi(X_2)}, 0,\ldots,0,x_1,\ldots,x_{r_3-1},z^{\phi(X_3)}0,\ldots,0),$$
where $x_h=1$, $1\leq h<\max\{r_1,\: r_2,\: r_3\}$. 

Similarly, we can prove that there exists a matrix representation
\begin{displaymath}
 Y_1+Y_2=\begin{bmatrix}\alpha_{1}^{'} \\ \vdots \\ \alpha_{m-m_2}^{'} \\ \beta_{1}^{'}\\ \vdots \\ \beta_{m-m_1}^{'} \\ \gamma_1^{'} \\ \vdots \\ \gamma_{m_1+m_2-m}^{'}\end{bmatrix},
\end{displaymath}
of $Y_1+Y_2$ such that
\begin{displaymath}
 Y_1=\begin{bmatrix}\alpha_{1}^{'} \\ \vdots \\ \alpha_{m-m_2}^{'} \\  \gamma_1^{'} \\ \vdots \\ \gamma_{m_1+m_2-m}^{'}\end{bmatrix}, ~~ Y_2=\begin{bmatrix}\beta_{1}^{'}\\ \vdots \\ \beta_{m-m_1}^{'} \\ \gamma_1^{'} \\ \vdots \\ \gamma_{m_1+m_2-m}^{'}\end{bmatrix},
\end{displaymath}
and
$(Y_1+Y_2)S\,{^t(Y_1+Y_2)}=$
$${\rm diag}(y_1,\ldots,y_{r_1-1},z^{\phi(Y_1)},0,\ldots,0,y_1,\ldots,y_{r_2-1},z^{\phi(Y_2)}, 0,\ldots,0,y_1,\ldots,y_{r_3-1},z^{\phi(Y_3)}0,\ldots,0),$$
where $y_h=1$, $1\leq h< \max\{r_1,\: r_2,\: r_3\}$ and $\phi(X_i)=\phi(Y_i),i=1,2,3$. 

By \cite[Lemma 6.8]{zwG}, there exists $T\in O_{2\nu+\delta,\Delta}\big(\mathbb{F}_q\big)$ such that $(X_1+X_2)T=Y_1+Y_2$, $X_1T=Y_1$ and $X_2T=Y_2$. So $X_1$---$X_2$ and $Y_1$---$Y_2$ are in the same orbit of $E\big(Oi\big(2\nu+\delta,q\big)\big)$ under the action of $O_{2\nu+\delta,\Delta}\big(\mathbb{F}_q\big)$.
\vskip 0.13cm

When $t(X_1)=t(Y_2),$ $t(X_2)=t(Y_1)$ and $t(X_1+X_2)=t(Y_1+Y_2),$ we conclude similarly that $X_1$---$X_2$ and $Y_1$---$Y_2$ are in the same orbit of $E\big(Oi\big(2\nu+\delta,q\big)\big)$ under the action of $O_{2\nu+\delta,\Delta}\big(\mathbb{F}_q\big)$.
\end{proof}

\vskip 0.3cm

\begin{Theorem} \label{4.3} Let $Oi\big(2\nu+\delta,q\big)$ be the orthogonal inner product graph over $\mathbb{F}_q$ and $X_1\text{---}X_2,Y_1\text{---}Y_2\in E\big(Oi\big(2\nu+\delta,q\big)\big).$ Then $X_1$---$X_2$ and $Y_1$---$Y_2$ are in the same orbit of $E\big(Oi\big(2\nu+\delta,q\big)\big)$ under the action of ${\rm Aut}\big(Oi\big(2\nu+\delta,q\big)\big)$ if and only if one of the following is true$:$ $(1)$ $t(X_1)=t(Y_1),$ $t(X_2)=t(Y_2),$ $t(X_1+X_2)=t(Y_1+Y_2);$ $(2)$ $t(X_1)=t(Y_2),$ $t(X_2)=t(Y_1),$ $t(X_1+X_2)=t(Y_1+Y_2).$
\end{Theorem}

\begin{proof} ($\Rightarrow$)  Suppose that $X_1$---$X_2$ and $Y_1$---$Y_2$ are in the same orbit of $E\big(Oi\big(2\nu+\delta,q\big)\big)$ under the action of ${\rm Aut}\big(Oi\big(2\nu+\delta,q\big)\big).$ Then there exists $\sigma\in{\rm Aut}\big(Oi\big(2\nu+\delta,q\big)\big)$ such that one of the following is true: $\sigma(X_1)=Y_1,$ $\sigma(X_2)=Y_2$; $\sigma(X_1)=Y_2,$ $\sigma(X_2)=Y_1.$ Without loss of generality we can assume that $\sigma(X_1)=Y_1$ and $\sigma(X_2)=Y_2$. Then
$t(X_1)=t(Y_1)$ and $t(X_2)=t(Y_2)$ by Theorem \ref{4.1}. It is sufficient to prove $t(X_1+X_2)=t(Y_1+Y_2)$. In order to prove $t(X_1+X_2)=t(Y_1+Y_2),$ we need only consider three cases: $\delta=0,1$ and 2. We first consider the case $\delta=0$. Let $\delta=0$. By Theorem \ref{ot1}, we know that there exist $k_1\in\mathbb{F}_q^{*2},k_2\ldots,k_{\nu}\in\mathbb{F}_q^*$ and $\pi\in{\rm Aut}(\mathbb{F}_q)$ such that $\sigma=\sigma_{(k_1,k_2,\ldots,k_{\nu},\pi)}.$
By Proposition \ref{4.2}, we know that $X_1+X_2$ and $\sigma_{\pi^{-1}}\sigma(X_1+X_2)$ are the same type orthogonal subspace. It is easy to check that $M(m,2s+\gamma,s,\Gamma)$ is cogredient to $\sigma_{\pi}(M(m,2s+\gamma,s,\Gamma))$. So $\sigma_{\pi^{-1}}\sigma(X_1+X_2)$ and $\sigma(X_1+X_2)$ are the same type orthogonal subspace. Thus
$X_1+X_2$ and $\sigma(X_1+X_2)=Y_1+Y_2$ are the same type orthogonal subspace.

Similar to the case of $\delta=0$, we can prove that
$X_1+X_2$ and $Y_1+Y_2$ are the same type orthogonal subspace when $\delta=1$ and $\delta=2$.

($\Leftarrow$) Suppose that one of the following is true: $(1)$ $t(X_1)=t(Y_1),$ $t(X_2)=t(Y_2),$ $t(X_1+X_2)=t(Y_1+Y_2)$; $(2)$ $t(X_1)=t(Y_2),$ $t(X_2)=t(Y_1),$ $t(X_1+X_2)=t(Y_1+Y_2).$

By Proposition~\ref{2.4}~and~\ref{4.1}, it is easy to check that there exist $T\in O_{2\nu+\delta,\Delta}(\mathbb{F}_q)$ and $\sigma_T\in {\rm Aut}\big(Oi\big(2\nu+\delta,q\big)\big)$ such that
$$\sigma_T(X_1)=X_1T=Y_1~~{\rm and}~~\sigma_T(X_2)=X_2T=Y_2$$
or
$$\sigma_T(X_1)=X_1T=Y_2~~{\rm and}~~\sigma_T(X_2)=X_2T=Y_1.$$
Thus $X_1$---$X_2$ and $Y_1$---$Y_2$ are in the same orbit of $E\big(Oi\big(2\nu+\delta,q\big)\big)$ under the action of ${\rm Aut}\big(Oi\big(2\nu+\delta,q\big)\big)$.
\end{proof}

{\small


\begin{thebibliography}{99}
\bibitem{dfapsl} D.F. Anderson, P.S. Livingston, The zero-divisor graph of a commutative ring, {\it J. Algebra,} {\bf217} (1999), 434--447.

\bibitem{jabusrm} J.A. Bondy, U.S.R. Murty, Graph Theory, {\it Graduate texts in mathematics, 244,} Springer, New York, 2008.
\bibitem{zgzw} Z.H. Gu, Z.X. Wan, Orthogonal graphs of odd characteristic and their automorphisms, {\it Finite Fields Appl.,} {\bf14} (2008), 291--313.

\bibitem{ztzw} Z.M. Tang, Z.X. Wan, J.M. Zhou, Symplectic graphs and their automorphisms, {\it European J. Combin.,} {\bf 27} (2006), 38--50.

\bibitem{dwxmjz} D. Wong, X.B. Ma, J.M. Zhou, The group of automorphisms of a zero-divisor graph based on rank one upper triangular matrices, {\it Linear Algebra Appl.,} {\bf 460} (2014), 242--258.

\bibitem{zwG} Z.X. Wan, Geometry of Classical Groups over Finite Fields, seconded., {\it Science Press,} Beijing, 2002.

\bibitem{zwkz} Z.X. Wan, K. Zhou, Orthogonal graphs of characteristic 2 and their automorphisms, {\it Sci. China Math.,} {\bf 52} (2009), 361--380.

\bibitem{zwkz2} Z.X. Wan, K. Zhou, Unitary graphs and their automorphisms, {\it Ann. Comb.,} {\bf 14} (2010), 367--395.


\bibitem{tw} T.S. Wu, On directed zero-divisor graphs of finite rings, {\it Discrete Math.,} {\bf 296} (2005), 73--86.

\end{thebibliography}
\end{document}